\documentclass[11pt]{article}
\usepackage{amsmath,amsfonts,amssymb,amsthm,mathrsfs}
\usepackage{bbm}
\usepackage{enumitem}
\usepackage{esint}
\usepackage[pagebackref,colorlinks,citecolor=blue,linkcolor=blue]{hyperref}
\usepackage[dvipsnames]{xcolor}
\usepackage{mathtools}  
\usepackage{combelow}

\usepackage{comment}

\usepackage{geometry}
\numberwithin{equation}{section}
\theoremstyle{plain}
\newtheorem{theorem}[equation]{Theorem}
\newtheorem{corollary}[equation]{Corollary}
\newtheorem{lemma}[equation]{Lemma}
\newtheorem{proposition}[equation]{Proposition}

\theoremstyle{definition}
\newtheorem{definition}[equation]{Definition}

\newtheorem{remark}[equation]{Remark}

\numberwithin{equation}{section}

\newcommand{\R}{{\mathbb R}}
\newcommand{\N}{{\mathbb N}}

\newcommand{\Om}{\Omega}

\providecommand{\vint}[1]{\mathchoice
	{\mathop{\vrule width 5pt height 3 pt depth -2.5pt
			\kern -9pt \kern 1pt\intop}\nolimits_{\kern -5pt{#1}}}
	{\mathop{\vrule width 5pt height 3 pt depth -2.6pt
			\kern -6pt \intop}\nolimits_{\kern -3pt{#1}}}
	{\mathop{\vrule width 5pt height 3 pt depth -2.6pt
			\kern -6pt \intop}\nolimits_{\kern -3pt{#1}}}
	{\mathop{\vrule width 5pt height 3 pt depth -2.6pt
			\kern -6pt \intop}\nolimits_{\kern -3pt{#1}}}}

\newcommand{\eps}{\varepsilon}
\newcommand{\loc}{\mathrm{loc}}

\newcommand\M{\operatorname{\mathcal M}}
\newcommand{\BV}{\mathrm{BV}}
\newcommand{\SBV}{\mathrm{SBV}}

\newcommand{\ch}{\text{\raise 1.3pt \hbox{$\chi$}\kern-0.2pt}}

\newcommand{\mres}{\mathbin{\vrule height 2ex depth 2.2pt width
		0.12ex\vrule height -0.3ex depth 2.2pt width .5ex}}

\DeclareMathOperator{\dive}{div}

\DeclareMathOperator*{\esssup}{ess\,sup}

\DeclareMathOperator{\Var}{Var}

\DeclareMathOperator{\osc}{osc}

\numberwithin{equation}{section}
\theoremstyle{remark} 
\theoremstyle{definition}

\newcommand{\B}{\mathbb{B}}

\newcommand{\BVY}{\mathbf{BVY}}

\newcommand{\EB}{\mathbf{E}}

\newcommand{\Le}{\mathcal{L}}

\newcommand{\SP}{\mathbb{S}}

\newcommand{\Y}{\mathbf{Y}}
\newcommand{\YM}{\mathbf{Y}^{\mathscr{M}}}

\AtBeginDocument{

}

\newcommand{\abs}[1]{\left\lvert #1 \right\rvert}
\newcommand{\abss}[1]{\lvert #1 \rvert}
\newcommand{\brac}[1]{\left(#1\right)}

\newcommand{\brangle}[1]{\left\langle#1\right\rangle}
\newcommand{\Brangle}[1]{\langle\kern-2.5pt\langle #1 \rangle\kern-2.5pt\rangle}

\newcommand{\col}{\colon\,}
\providecommand{\eps}{\varepsilon}
\newcommand{\id}{\,\mathrm{d}}

\newcommand{\norm}[1]{\left\lVert#1\right\rVert}

\newcommand{\normiii}[1]{{\left\vert\kern-0.25ex\left\vert\kern-0.25ex\left\vert #1 \right\vert\kern-0.25ex\right\vert\kern-0.25ex\right\vert}}

\newcommand{\weaksto}{\overset{*}{\rightharpoonup}}

\AtEndDocument{\bigskip{\footnotesize%

    	\noindent \textsc{Academy of Mathematics and Systems Science,\\
		Chinese Academy of Sciences\\
		Beijing 100190, PR China}\\
		\textit{E-mail address}, P. Lahti:  \texttt{panulahti@amss.ac.cn}\\

    \noindent\textsc{Max Planck Institute for Mathematics in the Sciences,\\ Inselstrasse 22, Leipzig, 04103, Germany}\\
    \textit{E-mail address}, Z.~Li: \texttt{zhuolin.li@mis.mpg.de}

}}

\begin{document}
	\title{Non-local functionals, total variation,
		and Gamma-convergence with respect to area-strict convergence
		\footnote{{\bf 2020 Mathematics Subject Classification}: 26B30, 26A46, 49J45
			\hfill \break {\it Keywords\,}: Total variation, special function of bounded variation,
			nonlocal functional, area-strict convergence, Gamma-convergence
	}}
	\author{Panu Lahti and Zhuolin Li}

	\maketitle
	
	\begin{abstract}
		We study a class of non-local functionals that was introduced by
		Brezis--Seeger--Van Schaftingen--Yung \cite{BSSY2},
		and can be used to characterize the total variation of functions.
		We establish the $\Gamma$-limit of these functionals with respect to area-strict convergence.
	\end{abstract}

	\section{Introduction}

%In the past two decades,
%there has been widespread interest in characterizing Sobolev and BV (bounded variation)
%functions by means of non-local functionals;
%see e.g. \cite{ABBF,AGMP,BBM,Bre,BN18,BN16,Dav,FMS}.
%One of these non-local functionals is defined as follows.
Let $n\in \N$ and let $\gamma\in\R$,
and define the measure
\[
\nu_{\gamma}(A):=\iint_A |x-y|^{\gamma-n}\id y\id x,
\quad A\subset \R^n\times \R^n.
\]
Let $\Om\subset \R^n$ be open. For a measurable function $u\colon \Om\to \R$,
let
\[
E_{\gamma,\lambda}(u,\Om)
:=\left\{(x,y)\in \Om\times \Om\colon
|u(x)-u(y)|>\lambda|x-y|^{1+\gamma}\right\}, \quad \lambda \in (0,\infty),
\]
and then define the functional
\begin{align*}
	F_{\gamma,\lambda}\left(u,\Om\right)
	:=\lambda\nu_\gamma (E_{\gamma,\lambda}(u,\Om)).
\end{align*}
This was studied in the case $\gamma=n$ in \cite{BVSY,Pol}, and then
more generally for $\gamma\in\R$ by Brezis--Seeger--Van Schaftingen--Yung in
\cite{BSSY,BSSY2}.

Denoting $x=(x_1,\ldots,x_n)\in \R^n$, let
\[
C_n:=\int_{\mathbb S^{n-1}}|x_1|\id \mathcal H^{n-1}(x),
\]
where $\mathcal H^{n-1}$ is the $(n-1)$-dimensional Hausdorff measure,
and $\mathbb S^{n-1}$ is the unit sphere.
We will always consider $\gamma>0$.
In response to a question posed in \cite{BSSY},
Picenni \cite{Pic} showed that
for a special function of bounded variation $u\in \SBV(\R^n)\subset \BV(\R^n)$, we have
\begin{equation}\label{eq:SBV limit}
\lim_{\lambda \to\infty}
F_{\gamma,\lambda}(u,\R^n)
= \frac{C_n}{\gamma}|D^au|(\R^n)
+\frac{C_n}{\gamma+1}|D^ju|(\R^n).
\end{equation} 
See $\S$\ref{subsec:def} for the notation. The space $\SBV(\Omega)$ of special functions of bounded variation is a subset of $\BV(\Omega)$, composed of functions with zero Cantor parts.
%where $D^au$ and $D^ju$ denote the absolutely continuous and jump parts
%of the variation measure, and the notation $u\in \SBV(\R^n)$ means that the Cantor part $D^c u$ vanishes.
%if $f\in L^1_{\loc}(\R^n)$ with $\Var(f,\R^n)<\infty$,
%and $\gamma>0$, then
%\[
%\liminf_{\lambda \to\infty}
%F_{\gamma,\lambda}(f,\R^n)
%\ge \frac{C_n}{\gamma}|D^af|(\R^n)
%+\frac{C_n\gamma }{2(1+2\gamma)(1+\gamma)}|D^cf|(\R^n)
%+\frac{C_n}{\gamma+1}|D^jf|(\R^n),
%\]
%where $|D^af|$, $|D^cf|$, and $|D^jf|$ denote the absolutely continuous, Cantor, and jump parts
%of the variation measure.
%The proof in \cite{Pic} is based on first considering the case $n=1$, and
%analyzing the blow-up behavior of $f$ at $|Df|$-a.e. point in $\R$;
%for $|D^cf|$-a.e. point,
%this blow-up behavior has a less neat representation than for
%$|D^a f|$-a.e. point and $|D^j f|$-a.e. point,
%and this leads to the somewhat small coefficient in front of $|D^cf|(\R^n)$.

%We analyze the contribution of $|D^cf|(\R)$
%using a different, more ``global'' strategy.
%Then we generalize to higher dimensions in the usual way.
The first author of the current paper showed the following lower bound in \cite{L-sharp}
for BV functions whose variation measure may also have a Cantor part:
	for $\Om\subset \R^n$ open and $u$ in the homogeneous BV space $\dot{\BV}(\Om)$, we have
	\begin{equation}\label{eq:sharp lower bound}
	\liminf_{\lambda \to\infty}
	F_{\gamma,\lambda}(u,\Om)
	\ge \frac{C_n}{\gamma}|D^au|(\Om)
	+\frac{C_n}{\gamma+1}|D^cu|(\Om)
	+\frac{C_n}{\gamma+1}|D^ju|(\Om).
	\end{equation}

For  BV functions whose variation measure $Du$ has a nonvanishing Cantor part $D^c u$, it may happen that
the limit $	\lim_{\lambda \to\infty}F_{\gamma,\lambda}(u,\Om)$ fails to exist, see
\cite[Example 3.23]{L-sharp}.
This is one reason why it is of interest to investigate
the \emph{$\Gamma$-limit}; see Section \ref{sec:definitions} for definitions.
The $\Gamma$-limit is currently not known, but if it exists, it is necessarily at most
\[
\frac{C_n}{\gamma+1}|Du|(\R^n),
\]
see Remark \ref{rmk:Gamma conv}.
Note that this limit does not agree,
in the case of $\SBV$ functions,
with the ``pointwise'' limit obtained in \eqref{eq:SBV limit}.
For this reason, we consider $\Gamma$-convergence
with respect to a different topology from the usual $L^1_{\loc}$-topology;
namely the \emph{area-strict topology}.
We call $\Gamma$-convergence with respect to the area-strict topology \emph{$\Gamma_{AS}$-convergence}.

The following is our main theorem.

\begin{theorem}\label{thm:Gamma area strict limit}
	Let $\gamma>0$ and let $\Om$ be either $\R^n$ or a bounded domain with
	Lipschitz boundary.
	The $\Gamma_{AS}$-limit of the functionals $F_{\gamma,\lambda}(\cdot,\Om)$ as $\lambda\to\infty$ is 
	\[
		\frac{C_n}{\gamma}|D^au|(\Om)+ \frac{C_n}{\gamma+1}|D^s u|(\Om).
	\]
\end{theorem}
Note that for $u\in\dot{\SBV}(\Om)$, this limit does agree with the ``pointwise'' limit of \eqref{eq:SBV limit}.
To obtain the lower bound of the $\Gamma_{AS}$-limit,
we first prove some results on the area-strict convergence of BV functions using
\emph{generalised Young measures} as a tool. After this, we are able to use
the lower bound \eqref{eq:sharp lower bound}, combined with
a (probably nonoptimal) $\Gamma$-convergence result from the recent paper \cite{MaGo}.
As is often the case, the upper bound is somewhat easier, and it requires 
constructing a sequence of functions for which the Cantor part is replaced by a jump part.
Both directions also utilize a commonly used slicing argument, in which the problem
is essentially reduced to the one-dimensional case.
\\

	\textbf{Acknowledgments.}
	The authors wish to thank Jan Kristensen for helpful discussions,
    and the anonymous referees for their careful reading and helpful comments, which have greatly improved the manuscript.

%\begin{theorem}\label{thm:main}
%	Let $\gamma>0$, let $\Om\subset \R^n$ be open, and let $f\in L^1_{\loc}(\Om)$ with
%	$\Var(f,\Om)<\infty$. Then
%	\[
%	\liminf_{\lambda \to\infty}
%	F_{\gamma,\lambda}(f,\Om)
%	\ge \frac{C_n}{\gamma}|D^af|(\Om)
%	+\frac{C_n}{\gamma+1}|D^cf|(\Om)
%	+\frac{C_n}{\gamma+1}|D^jf|(\Om).
%	\]
%\end{theorem}

\section{Definitions}\label{sec:definitions}

\subsection{Basic notations and definitions}\label{subsec:def}

We work in the Euclidean space $\R^n$ with $n\ge 1$.
The Euclidean norm is denoted by $|\cdot|$, and
the $n$-dimensional Lebesgue outer measure by $\Le^n$. 
By $\Om$ we always mean an open subset of $\R^n$.
We denote the characteristic function of a set $A\subset\R^n$ by $\mathbbm{1}_A\colon \R^n\to \{0,1\}$.
The sets of finite non-negative Radon measures and probability measures on an open set $\Om\subset \R^n$ are
denoted by $\M^+(\Om)$ and $\M^1(\Om)$, respectively.
Given $l\in\N$, the set of $\R^l$-valued Radon measures,
possibly unbounded, is denoted by $\M(\Om;\R^l)$.
Weak* convergence $\mu_i\overset{*}{\rightharpoondown}\mu$ in $\M(\Om;\R^l)$ means
$\int_\Om\phi\id \mu_i\to \int_\Om\phi\id \mu$ for any
continuous function $\phi$ that is compactly supported in $\Om$ (in short, $\phi\in C_c(\Om)$).
For any
$\mu \in \M^+(\R^l)$ with
$\int_{\R^l} |z|\id \mu(z)<\infty$, its barycentre is defined by
    \[ 
    \bar{\mu} := \int_{\R^l} z\id \mu(z).
    \]

The theory of $\BV$ functions that we rely on can be found in the monograph
Ambrosio--Fusco--Pallara \cite{AFP}.
A function
$u\in L^1(\Omega)$ is of bounded variation,
denoted $u\in \BV(\Omega)$, if its weak derivative
is an $\R^{n}$-valued Radon measure with finite total variation. This means that
there exists a (necessarily unique) $\R^n$-valued Radon measure $Du$ 
such that for all $\varphi\in C_c^1(\Omega)$, the integration-by-parts formula
\[
\int_{\Omega}u\frac{\partial\varphi}{\partial x_k}\id x
=-\int_{\Omega}\varphi\id D_ku ,\quad k=1,\ldots,n,
\]
holds.
The total variation of $Du$ is denoted by $|Du|$.
If we do not know a priori that
$u\in L^1_{\loc}(\Om)$
is a BV function, then we consider
\[
\Var(u,\Om):=\sup\left\{\int_{\Om}u\dive\varphi\id x,\,\varphi\in C_c^{1}(\Om;\R^{n}),
\,|\varphi|\le 1\right\}.
\]
If $\Var(u,\Om)<\infty$, then the $\R^{n}$-valued Radon measure $Du$
exists and $\Var(u,\Om)=|Du|(\Om)$
by the Riesz representation theorem, and we denote $u\in \dot{\BV}(\Om)$.
Moreover, $u\in\BV(\Om)$ provided that also $u\in L^1(\Om)$.
The BV norm is defined by
\[
\Vert u\Vert_{\BV(\Om)}:=\Vert u\Vert_{L^1(\Om)}+\Var(u,\Om).
\]

Denote by $S_u\subset \Om$ the set of non-Lebesgue points of $u\in \BV_{\loc}(\Om)$.
We write the Radon-Nikodym decomposition of the variation measure of $u$ into the absolutely continuous and singular parts with respect to $\Le^n$
as
\[
Du=D^a u+D^s u.
\]
We denote the density of $D^a u$ by $\nabla u$.
Furthermore, we define the jump and Cantor parts of $Du$ as
\[
D^j u\coloneqq D^s u\mres S_u, \qquad D^c u\coloneqq  D^s u\mres (\Om\setminus S_u),
\]
where
\[
D^s u \mres S_u(A):=D^s u (S_u\cap A),\quad \textrm{for } D^s u\textrm{-measurable } A\subset \Om.
\]
Thus we get the decomposition
\[
Du=D^a u+ D^c u+ D^j u.
\]
If $\Var(u,\Om)=\infty$, we interpret the quantity $|Du|(\Om)$, or a sum such as
	\[
\frac{C_n}{\gamma}|D^au|(\Om)+ \frac{C_n}{\gamma+1}|D^s u|(\Om),
\]
to be infinite.

	\begin{definition}\label{def:strict conv}
	Let $u,u_\lambda\in L^1_{\loc}(\Om)$ for $\lambda>0$.
	We say that $u_\lambda\to u$ strictly in $\Om$ if $u_{\lambda}\to u$ in $L^1_{\loc}(\Om)$ and 
	\[
	\lim_{\lambda\to\infty}|D u_\lambda|(\Om)
	=|D u|(\Om).
	\]
	\end{definition}
    \par Denote the function $(1+\abss{\cdot}^2)^{\frac{1}{2}}$ by $E(\cdot)$.

	\begin{definition}\label{def:area strict}
	Let $u,u_\lambda\in L^1_{\loc}(\Om)$ for $\lambda>0$.
	We say that $u_\lambda\to u$ area-strictly in $\Om$ if $u_{\lambda}\to u$ in $L^1_{\loc}(\Om)$ and 
	\[
	\lim_{\lambda\to\infty}\int_{\Om}E(Du_{\lambda})
	=\int_{\Om}E(Du),
	\] where for any $v \in \dot{\BV}(\Om)$ the integral is defined by
    \[
    %E(Dg)(\Omega):=
    \int_{\Omega} E(Dv) := \int_{\Omega} E(\nabla v)\id x + \abss{D^s v}(\Omega).
    \]
	\end{definition}
	
	\subsection{$\Gamma$-convergence}
	
	The usual notion of $\Gamma$-convergence in this context is defined with respect to $L^1_{\loc}$-convergence, as follows.
	As usual, $\Om\subset \R^n$ is an open set.
	
	\begin{definition}
		Functionals $\Phi_{\lambda}\colon L^1_{\loc}(\Om)\to [0,\infty]$, with $\lambda>0$,
		are said to $\Gamma$-converge to a functional $\Phi\colon L^1_{\loc}(\Om)\to [0,\infty]$
		as $\lambda\to \infty$
		if the following two properties hold:
		\begin{enumerate}
			\item For every $u\in L_{\loc}^1(\Om)$ and for every family $\{u_\lambda\}_{\lambda>0}$ such that
			$u_\lambda\to u$ in $L^{1}_{\loc}(\Om)$ as $\lambda\to \infty$, we have
			\[
			\liminf_{\lambda\to \infty}\Phi_{\lambda}(u_{\lambda},\Om)\ge \Phi(u,\Om).
			\]
			\item For every $u\in L_{\loc}^1(\Om)$, there exists a family $\{u_\lambda\}_{\lambda>0}$ such that
			$u_{\lambda}\to u$ in $L_{\loc}^1(\Om)$ as $\lambda\to \infty$, and
			\[
			\limsup_{\lambda\to \infty}\Phi_{\lambda}(u_{\lambda},\Om)\le \Phi(u,\Om).
			\]
		\end{enumerate}
	\end{definition}

	We also consider $\Gamma$-convergence defined with respect to area-strict convergence, as follows.
			
	\begin{definition}
		Functionals $\Phi_{\lambda}\colon L^1_{\loc}(\Om)\to [0,\infty]$, with $\lambda>0$,
		are said to $\Gamma_{\mathrm{AS}}$-converge to a functional
		$\Phi\colon L^1_{\loc}(\Om)\to [0,\infty]$
		as $\lambda\to \infty$
		if the following two properties hold:
		\begin{enumerate}
			\item For every $u\in L^1_{\loc}(\Om)$ and for every family $\{u_\lambda\}_{\lambda>0}$ such that
			$u_\lambda\to u$ area-strictly as $\lambda\to \infty$, we have
			\[
			\liminf_{\lambda\to \infty}\Phi_{\lambda}(u_{\lambda},\Om)\ge \Phi(u,\Om).
			\]
			\item For every $u\in L^1_{\loc}(\Om)$, there exists a family $\{u_\lambda\}_{\lambda>0}$ such that
			$u_\lambda\to u$ area-strictly as $\lambda\to \infty$, and
			\[
			\limsup_{\lambda\to \infty}\Phi_{\lambda}(u_{\lambda},\Om)\le \Phi(u,\Om).
			\]
		\end{enumerate}
	\end{definition}

As mentioned before, the existence of a $\Gamma$-limit (defined with respect to $L^1_{\loc}$-convergence)
of the functionals $F_{\gamma,\lambda}(\cdot,\Om)$ (recall that we always assume $\gamma>0$)
is  currently not known.
Our main theorem, Theorem \ref{thm:Gamma area strict limit},
says that the $\Gamma_{AS}$-limit of the functionals $F_{\gamma,\lambda}(\cdot,\Om)$ 
exists
and is given by
\[
\frac{C_n}{\gamma}|D^au|(\Om)+ \frac{C_n}{\gamma+1}|D^s u|(\Om).
\]
In Sections \ref{sec:lower bound} and \ref{sec:upper bound},
we will prove the lower and upper bounds, respectively, of the $\Gamma_{AS}$-convergence.

\subsection{Generalised Young measures}\label{subsec: Young}
    \par In this subsection, we introduce the notion of generalised Young measures and some
    related results, which will be helpful in the subsequent proofs. For a systematic discussion
    on this topic, see \cite{Rindler, KR20}.
    \par The set $\Omega\subset \R^n$ is assumed to be a bounded open set with $\Le^n(\partial \Om)=0$ throughout this subsection.
    Fix a positive integer $m$,
    and denote by $\B^m$ and $\SP^{m-1}$ the open unit ball and the unit sphere in $\R^m$,
    respectively. Consider the linear transformation $S$ on $C(\overline{\Omega}\times \B^m)$
    given by 
    \[ 
    (Sf)(x,z):= (1-\abss{z})f\brac{x,\frac{z}{1-\abss{z}}}, 
    \quad x\in \overline{\Omega},\, z\in \B^m, 
    \] 
    for any $f \in C(\overline{\Omega}\times \R^m)$. 
    Then set 
    \[ 
    \EB(\Omega;\R^m) := \{f \in C(\overline{\Omega}\times \R^m)\col Sf \in C(\overline{\Omega}\times \overline{\B}^m) \},
    \] where by $Sf \in C(\overline{\Omega}\times \overline{\B}^m)$ we mean that $Sf$ can be continuously extended to $\overline{\Omega}\times \partial \B^m$. % to the boundary.
    The norm on $\EB (\Omega;\R^m)$ is given by the natural choice
    \[
    \norm{f}_{\EB(\Omega;\R^m)}:= \norm{Sf}_{C(\overline{\Omega}\times \overline{\B}^m)} 
    = \sup_{(x,z)\in \overline{\Omega}\times \overline{\B}^m} \abss{(Sf)(x,z)},
    \]
    under which $\EB(\Omega;\R^m)$ is a Banach space and
    $S \col \EB(\Omega;\R^m) \to C(\overline{\Omega}\times \overline{\B}^m)$ is
    an isometric isomorphism (\cite{Rindler}, \S 12.1 and \cite{KR20}, \S 2).
	\par If $f \in \EB(\Omega;\R^m)$, the {\it recession function} defined by
		\begin{equation}\label{eq:lnflrecession}
		f^{\infty}(x,z) := 
        \lim_{\substack{(x',z')\to (x,z)\\ t \to \infty}} \frac{f(x',tz')}{t}
		\end{equation}
        exists for any $(x,z) \in \overline{\Omega}\times \R^m$. The function $f^{\infty}$ is
        positively $1$-homogeneous in the second argument and coincides with $Sf$ on
        $\overline{\Omega}\times \partial \B^m$. The class $\EB(\R^m)$ consists of autonomous
        integrands (i.e., integrands only depending on $z \in \R^m$) and is defined analogously.
    \par For any integrand $f \in \EB(\Om;\R^m)$, we can define the corresponding integral
    on any $\R^m$-valued Radon measure $\mu\in\M(\Om;\R^m)$
    by
    \[
    %f(\cdot, \mu)(A):=
    \int_{A}f(x,\mu) := \int_A f\brac{\cdot ,\frac{\id \mu}{\id \Le^n}}\id \Le^n + \int_Af^{\infty}\brac{\cdot ,\frac{\id \mu}{\id\abss{\mu^s}}}\id \abss{\mu^s}
    \]
    for any Borel set $A \Subset \Om$. This is actually the area-strictly continuous extension of $\int f(x,\nabla u)\id x$ on $W^{1,1}(\Omega; \R^m)$ by the Reshetnyak continuity theorem (\cite{Res} and \cite[Theorem 2.39]{AFP}) and it is consistent with the quantity $\int_{\Om} E(Dg)$ in Definition \ref{def:area strict}.

    \par Now we give the definition of generalised Young measures.
    Recall that in this subsection, $\Omega \subset \R^n$ is a bounded open set. 
    \begin{definition}
	An $\R^m$-valued
    \emph{generalised Young measure} defined on $\Omega$ is a triple $\nu=((\nu_x)_{x\in \Omega}, \lambda_{\nu},(\nu_x^{\infty})_{x\in \overline{\Omega}})$, where
		\begin{enumerate}[label=(\roman*)]
		\item $(\nu_x)_{x\in \Omega} \subset \M^1 (\R^m)$ is the {\it oscillation measure};
		\item $\lambda_{\nu} \in \M^+(\overline{\Omega})$ is the {\it concentration measure};
		\item $(\nu_x^{\infty})_{x\in \overline{\Omega}}\subset \M^1(\SP^{m-1})$ is the {\it concentration-direction measure};
		\end{enumerate} and it satisfies:
		\begin{enumerate}[label=(\roman*)]\setcounter{enumi}{3}
		\item the map $x\mapsto \nu_x$ is weakly$^{\ast}$ measurable with respect to $\Le^n \mres \Omega$, i.e., the function $x\mapsto \brangle{\nu_x, f(x,\cdot)}:=\int_{\R^m}f(x,\cdot)\id \nu_x$
        is $\Le^n$-measurable for any bounded Borel function $f\col \Omega\times \R^m \to \R$;
		\item the map $x\mapsto \nu^{\infty}_x$ is weakly$^{\ast}$ measurable with respect to $\lambda_{\nu}$, i.e., the function $x\mapsto \brangle{\nu^{\infty}_x, f(x,\cdot)}$ is $\lambda_{\nu}$-measurable for any bounded Borel function $f\col \overline{\Omega}\times \SP^{m-1} \to \R$;
		\item $x\mapsto \brangle{\nu_x,\abss{\cdot}} \in L^1(\Omega)$, i.e., the $1$-moment $\int_{\Omega}\int_{\R^m}\abss{z}\id \nu_x \id x$ is finite.
		\end{enumerate} The collection of all $\R^m$-valued generalised Young measures on $\Omega$ is denoted by $\YM(\Omega;\R^m)$.
	\end{definition}
    \par The functions in $\EB(\Omega;\R^m)$ can be considered as ``test integrands" for Young measures
    (for simplicity, we sometimes omit the word ``generalised''). For $f\in \EB(\Omega;\R^m)$ and $\nu \in \YM(\Omega;\R^m)$, the duality pairing is defined by
	\[
	\Brangle{\nu,f} := \int_{\Omega}\brangle{\nu_x, f(x,\cdot)}\id x + \int_{\overline{\Omega}} \brangle{\nu_x^{\infty},f^{\infty}(x,\cdot)}\id \lambda_{\nu}(x).
	\]
    Correspondingly, we have $\YM(\Omega;\R^m) \subset (\EB(\Omega;\R^m))^{\ast}$.

    \begin{definition}
	\begin{enumerate}[label=(\roman*)]
	\item Given any finite Radon measure $\gamma \in \M(\overline{\Omega};\R^m)$ with the Lebesgue-Radon-Nikod\'{y}m decomposition 
    \[ \gamma = \gamma^{a}\Le^n \mres \Omega + \gamma^s, \]
    we associate to it an {\it elementary Young measure} $\varepsilon_{\gamma}\in \YM(\Omega;\R^m)$, which is given by
		\begin{equation}
		\varepsilon_{\gamma}= ((\delta_{\gamma^{a}(x)})_{x\in \Omega},\abss{\gamma^s},(\delta_{P(x)})_{x\in \overline{\Omega}}),
		\end{equation} where $\delta_z$ is the Dirac measure supported at $z\in \R^m$ and $P:= \frac{\id \gamma^s}{\id \abss{\gamma^s}}$.
	\item Given a sequence $\{\gamma_j\}_{j\in \N}\subset \M(\overline{\Omega};\R^m)$
    with $\sup_j \abss{\gamma_j}(\overline{\Omega})<\infty$, we say that $\{\gamma_j\}$
    {\it generates} $\nu\in \YM(\Omega;\R^m)$ (written as $\gamma_j \overset{\Y}{\to} \nu$)
    if $\varepsilon_{\gamma_j} \overset{\ast}{\rightharpoonup} \nu$ in $\YM(\Omega;\R^m)$,
    i.e., 
    \[
    \Brangle{\varepsilon_{\gamma_j},f}\to \Brangle{\nu,f} \quad \mbox{as }j\to \infty
    \] 
    for any $f \in \EB(\Omega;\R^m)$.
	\end{enumerate}
	\end{definition} 

    \par The following result, which is considered to be the fundamental theorem of
    generalised Young measures, implies the precompactness of
    $\M(\overline{\Omega};\R^m)\subset\YM(\Omega;\R^m)$ under the weak$^{\ast}$ topology.
    Recall that $\Omega\subset \R^n$ is assumed to be a bounded open set throughout this subsection.
	\begin{theorem}[\cite{Rindler}, Theorem 12.5, \cite{AB}, Theorem 5]\label{thm:YMft}
	 Suppose that $\{\gamma_j\}_{j\in \N} \subset \M(\overline{\Omega};\R^m)$ is a sequence
     of Radon measures with $\sup_j \abss{\gamma_j}(\overline{\Omega})<\infty$. Then there
     exists a subsequence $\{\gamma_{j_k}\}_{k\in \N}$ and a Young measure
     $\nu \in \YM(\Omega;\R^m)$ such that $\gamma_{j_k} \overset{\Y}{\to}\nu$.
	\end{theorem}

    \par Given any $u\in \BV(\Omega)$, we can associate to
    $Du \in \M(\Omega;\R^n)$ the elementary Young measure $\varepsilon_{Du}$ defined as above,
    and define the class of BV Young measures: 
		\begin{multline*}
		\BVY(\Omega;\R^n):= \{\nu \in \YM(\Omega;\R^n)\col \mbox{ there exists } \{u_j\}
        \subset \BV(\Omega)  \mbox{ with } Du_j \overset{\Y}{\to}\nu \}.
		\end{multline*}
	Then Theorem \ref{thm:YMft} implies the following:
    \begin{corollary}\label{cor:YMft}
	Suppose that $\{u_j\}_{j\in \N} \subset \BV(\Omega)$ is bounded.
    Then there exists a subsequence $\{u_{j_k}\}_{k\in \N}$ and $\nu \in \BVY(\Omega)$
    such that
    \[ 
    Du_{j_k}\overset{\Y}{\to}\nu .
    \]
	\end{corollary}

    \par The representation of limits via Young measures can be extended to a larger class of integrands. A function $f: \overline{\Om} \times \R^m \to \R$ is Carath\'{e}odory if $f(\cdot\, , z)$ is measurable for every $z\in \R^m$ and $f(x,\cdot\,)$ is continuous for $\Le^n$-a.e. $x\in \overline{\Omega}$. Such a function $f$ is said to be of linear growth if $\abss{f(x,z)} \leq L(1+\abss{z})$ for some $L>0$ and any $(x,z) \in \overline{\Omega}\times \R^m$. Then define
        \begin{multline*}
            \mathbf{R}(\Omega;\R^m):= \{ f: \overline{\Om} \times \R^m \to \R\col f \mbox{ is Carath\'{e}odory and of linear growth,} \\\mbox{ and }f^{\infty}\in C(\overline{\Omega}\times \R^m) \mbox{ exists} \}.
        \end{multline*}

    \begin{proposition}[\cite{Rindler}, Proposition 12.11] \label{prop:YM Caratheodory}
        Suppose that $\nu_j \weaksto \nu$ in $\YM(\Omega;\R^m)$ and assume that
            $f \in \mathbf{R}(\Om;\R^m)$.
       Then $\Brangle{\nu_j,f}\to \Brangle{\nu,f}$.
    \end{proposition}
%    \begin{proposition}[\cite{Rindler}, Proposition 12.11] \label{prop:YM Caratheodory}
%        Suppose that $\nu_j \weaksto \nu$ in $\YM(\Omega,\R^m)$ and assume either
%        \begin{enumerate}[label=(\alph*)]
%            \item $f \in \mathbf{R}(\Om;\R^m)$ or
%            \item $f(x,z)=\mathbbm{1}_{B}(x)g(x,z)$, where $g \in \EB(\Omega;\R^m)$ and $B \subset \overline{\Om}$ is a Borel set with $\Le^n (\partial B)=0=\lambda_{\nu}(\partial B)$.
%        \end{enumerate} Then we have $\Brangle{\nu_j,f}\to \Brangle{\nu,f}$.
%    \end{proposition}
    
\section{Lower bound}\label{sec:lower bound}
	
	In this section we prove the lower bound for the $\Gamma_{AS}$-convergence of the functionals
	$F_{\gamma,\lambda}(\cdot,\Om)$.

\subsection{Area-strict and norm convergence}

First we will investigate the relationship between area-strict convergence and convergence
in the BV norm, using generalised Young measures as a tool.
The results of this subsection will only be used in the case $n=1$, but we prove them for general $n$.

Recall that $E(\cdot)$ denotes the function $(1+\abss{\cdot}^2)^{\frac{1}{2}}$. 

	\begin{lemma}\label{lem:Eseparation}
    The function $h(t):= E(t)-t$ defined on $[0,\infty)$ is strictly decreasing.
    In particular, we have
    \[
    E(t)-t \geq E(s)-s \quad \mbox{for any } 0\le t\le s<\infty,
    \]
    and
    equality holds if and only if $t=s$.
    \end{lemma}
    \begin{proof}
    We have
    \[
    h^{\prime}(t) = E^{\prime}(t)-1 =\frac{t}{E(t)}-1 <0
    \quad \mbox{for all } 0\le t<\infty.
    \]
    \end{proof}

\begin{lemma}\label{lem:areastrYM}
Suppose that $\Omega \subset \R^n$ is a bounded open set with $\Le^n(\partial \Omega)=0$,
and a sequence $\{u_j\} \subset \dot{\BV}(\Omega)$ converges to $u\in \dot{\BV}(\Omega)$
in the area-strict sense.
Then $\{D u_j\}$ generates a Young measure $\nu = ((\nu_x)_{x\in \Omega}, \lambda_{\nu}, (\nu_x^{\infty})_{x\in \overline{\Omega}}) \in Y^{\mathscr{M}}(\Omega;\R^n)$, which satisfies the following:
    \begin{align}
        &\nu_x = \delta_{\bar{\nu}_x}\quad  \mbox{for } \Le^n\mbox{-a.e. }x\in \Omega,\nonumber\\
        &\lambda_{\nu} \perp \Le^n,\quad  \lambda_{\nu}\mres \partial \Omega =0,\nonumber\\
        &\nu_x^{\infty} = \delta_{\bar{\nu}_x^{\infty}}\quad \mbox{for } \lambda_{\nu}\mbox{-a.e. }x\in \Omega.\label{eq:ASYM3}
    \end{align}
    Moreover, there hold
    \begin{align}
        &\nabla u(x) = \bar{\nu}_x\quad \mbox{for } \Le^n\mbox{-a.e. }x\in \Omega, \label{eq:ASYM4}\\
        &D^s u = \bar{\nu}_{\cdot}^{\infty}\lambda_{\nu}.\notag
        %,\quad \mbox{for }\lambda_{\nu}\mbox{-a.e. }x\in \Omega.
    \end{align}
\end{lemma}
Note that this is essentially the autonomous version of Reshetnyak's continuity theorem
(see \cite{KrRind2}, Theorem 5).
\begin{proof}
    Since $\abss{Du_j}(\Om)$ is bounded, by Theorem \ref{thm:YMft}, there is a subsequence (not relabelled) that generates a Young measure $\nu \in Y^{\mathscr{M}}(\Omega;\R^n)$.
    Up to a further subsequence, we can assume $Du_j \overset{*}{\rightharpoondown} Du$ (see e.g.
    \cite[Proposition 3.13]{AFP}).
    Then for any $\varphi \in C_c^{\infty}(\Omega)$, the following holds true:
    \begin{align*}
       %  \brangle{Du, \varphi} 
       \int_{\Omega}\varphi \,dDu
        %&= \lim_{j\to \infty}\brangle{Du_j, \varphi} 
        &=  \lim_{j\to \infty}\int_{\Omega}\varphi \,dDu_j \\
        &= \int_{\Omega}\varphi(x) \brangle{\nu_x, \, \cdot}\id x + \int_{\overline{\Omega}} 
        \varphi(x)  \brangle{\nu^{\infty}_x, \, \cdot}\id \lambda_{\nu}(x) \\
        &= \int_{\Omega} \varphi(x)(\bar{\nu}_x \id x + \bar{\nu}^{\infty}_x \id \lambda_{\nu}(x)),
    \end{align*} which implies 
    \[
        Du %=\bar{\nu} 
        = \bar{\nu}_{\cdot} \Le^n \mres \Omega + \bar{\nu}_{\cdot}^{\infty} \lambda_{\nu}\mres \Omega.
    \]
    Equating the absolutely continuous and singular parts with respect to $\Le^n$
    on both sides, we obtain
    \begin{align*}
        &\nabla u(x) = \bar{\nu}_x + \bar{\nu}_x^{\infty}\frac{\id \lambda_{\nu}}{\id \Le^n}(x)
        \quad \mbox{for } \Le^n\mbox{-a.e. }x\in \Omega,\\
        &D^s u = \bar{\nu}_{\cdot}^{\infty}\lambda_{\nu}^s,
        %,\quad \mbox{for }\lambda_{\nu}^s\mbox{-a.e. }x\in \Omega,
    \end{align*}
    where $\lambda_{\nu}^s$ is the singular part of $\lambda_{\nu}$ with respect to $\Le^n$.
    Since $\nu$ is generated by $\{Du_j\}$, we have  
    \begin{align}
        \lim_{j\to \infty} \int_{\Omega} E(Du_j) &= \int_{\Omega}\brangle{\nu_x, E(\cdot)}\id x + \int_{\overline{\Omega}} \brangle{\nu_x^{\infty},\abss{\cdot}}\id \lambda_{\nu}(x) \label{eq:ASYMmid1}\\
        &\geq \int_{\Omega} E(\bar{\nu}_x)\id x + \int_{\Omega} \abss{\bar{\nu}_x^{\infty}}\frac{\id \lambda_{\nu}}{\id \Le^n}\id x +\int_{\Omega} \abss{\bar{\nu}_x^{\infty}}\id \lambda^s_{\nu}(x) + \lambda_{\nu}(\partial \Omega)\notag\\
        &\geq \int_{\Omega} E(\bar{\nu}_x +\bar{\nu}_x^{\infty}\frac{\id \lambda_{\nu}}{\id \Le^n})\id x +\int_{\Omega} \abss{\bar{\nu}_x^{\infty}}\id \lambda^s_{\nu}(x)+ \lambda_{\nu}(\partial \Omega)\notag\\
        & = \int_{\Omega} E(\nabla u)\id x + \abss{D^s u}(\Omega) + \lambda_{\nu}(\partial \Omega),\notag
    \end{align} where the second line follows from Jensen's inequality and the third from Lemma \ref{lem:Eseparation}. On the other hand, the sequence $\{u_j\}$ converges to $u$ in the area-strict sense, which indicates 
    \begin{equation} \label{eq:ASYMmid2}
    \lim_{j\to \infty} \int_{\Omega} E(Du_j) = \int_{\Omega} E(Du) = \int_{\Omega} E(\nabla u)\id x + \abss{D^s u}(\Omega).
    \end{equation} 
    Combining \eqref{eq:ASYMmid1} and \eqref{eq:ASYMmid2}, we know that
    $\lambda_{\nu}\mres \partial \Omega = 0$ and all the inequalities in \eqref{eq:ASYMmid1}
    are equalities. By the
    %strong
    strict convexity of $E(\cdot)$ and $\abss{\cdot}$ on $\SP^{n-1}$,
    the first inequality of \eqref{eq:ASYMmid1} --- which is now an equality --- implies
     \begin{align*}
        &\nu_x = \delta_{\bar{\nu}_x}\quad  \mbox{for } \Le^n\mbox{-a.e. }x\in \Omega,\\
        &\nu_x^{\infty} = \delta_{\bar{\nu}_x^{\infty}}\quad \mbox{for } \lambda_{\nu}\mbox{-a.e. }x\in \Omega.
    \end{align*} These give the first and third equality in \eqref{eq:ASYM3}. From Lemma \ref{lem:Eseparation}, we know that the second inequality in \eqref{eq:ASYMmid1} holds true as an equality if and only if
    \[\bar{\nu}_x^{\infty}\frac{\id \lambda_{\nu}}{\id \Le^n}(x) = 0 \quad \mbox{for }\Le^n\mbox{-a.e. }x\in \Omega. \]
    Notice that
    the third line of \eqref{eq:ASYM3}
    indicates $\bar{\nu}_x^{\infty} \in \SP^{n-1}$ for $\lambda_{\nu}$-a.e. $x\in \Omega$, which further implies $\frac{\id \lambda_{\nu}}{\id \Le^n}(x) = 0$ for $\Le^n$-a.e. $x\in \Omega$,
    so that $\lambda_{\nu} \perp \Le^n$.
    \par The above argument implies that an arbitrary Young measure generated by some subsequence of $\{Du_j\}$ satisfies \eqref{eq:ASYM3} and \eqref{eq:ASYM4}, and then it is easy to see that $\{Du_j\}$ itself generates $\nu$ as in the statement. The proof is now complete.
\end{proof}

%\par Now we are ready to proof Proposition \ref{prop:area strict and norm}.
	\begin{proposition}\label{prop:area strict and norm}
		Let $\Om\subset \R^n$ be a bounded open set with $\Le^n(\partial \Omega)=0$.
		Let $\{u_\lambda\}_{\lambda>0}\subset \dot{\BV}(\Om)$
		and let $u\in \dot{\BV}(\Om)$ such that $u_{\lambda}\to u$ as $\lambda\to \infty$ area-strictly in $\Om$.
		Then we have
		\[
		\limsup_{\lambda\to\infty} |D(u_\lambda-u)|(\Om)\le 2|D^s u|(\Om).
		\]
	\end{proposition}
\begin{proof}%[Proof of Proposition \ref{prop:area strict and norm}]
From Lemma \ref{lem:areastrYM} we know that $Du_{\lambda} \overset{\mathbf{Y}}{\rightarrow} \nu$,
where $\nu\in Y^{\mathscr{M}}(\Omega;\R^n)$ is a $\BV$ Young measure and satisfies the following:
    \begin{align*}
        &\nu_x = \delta_{\bar{\nu}_x} \quad \mbox{and}\quad \nabla u(x) = \bar{\nu}_x\  \  \mbox{for } \Le^n\mbox{-a.e. }x\in \Omega,\\
        &\lambda_{\nu} \perp \Le^n,\quad  \lambda_{\nu}\mres \partial \Omega =0,\\
        &\nu_x^{\infty} = \delta_{\bar{\nu}_x^{\infty}} \ \  \mbox{for } \lambda_{\nu}\mbox{-a.e. }x\in \Omega, \quad \mbox{and} \quad 
        D^s u = \bar{\nu}_{\cdot}^{\infty}\lambda_{\nu} .
    \end{align*}
    Now consider the integrand $g(x,z):=\abss{z-\nabla u(x)}$. The integrand $g$ is Carath\'{e}odory
    by definition, but not necessarily in the class $\mathbf{R}(\Om;\R^n)$ since $\nabla u$ can be
    unbounded and the recession function of $g$ may not exist. Hence, replace $\nabla u$ by its truncation
    $T_M(\nabla u)$ defined by 
    \[
    T_M(\nabla u)(x):= \left\{ \begin{aligned}
    &\nabla u(x),& &\abss{\nabla u(x)}\leq M\\
    &M \frac{\nabla u(x)}{\abss{\nabla u(x)}},& &\abss{\nabla u(x)}>M.
    \end{aligned}\right.
    \]
    Then the corresponding integrand $g_M(x,z) := \abss{z-T_M(\nabla u)(x)}$ is Carath\'{e}odory and of linear growth, and $g^{\infty}_M(x,z)=\abss{z}$. By Proposition \ref{prop:YM Caratheodory}, we have
    \begin{align*}
       &\  \limsup_{\lambda\to \infty}\abss{Du_{\lambda} -Du}(\Omega) \leq  \limsup_{\lambda\to \infty}\abss{Du_{\lambda} -\nabla u \Le^n}(\Omega)+ \abss{D^s u}(\Omega)\\
       \leq&\ \limsup_{\lambda\to \infty}\abss{Du_{\lambda} -T_M(\nabla u)\Le^n}(\Omega)+\int_{\Omega}\abss{T_M(\nabla u)-\nabla u}\id x+ \abss{D^s u}(\Omega)\\
        = &\ \int_{\Omega} \brangle{\delta_{\bar{\nu}_x - T_M(\nabla u)(x)}, \, \abss{\cdot}}\id x + \int_{\Omega} \brangle{\delta_{\bar{\nu}_x^{\infty}},\, \abss{\cdot}}\id \lambda_{\nu}(x) +\int_{\Omega}\abss{T_M(\nabla u)-\nabla u}\id x+\abss{D^su}(\Omega)\\
        =&\ 2\int_{\Omega}\abss{T_M(\nabla u)-\nabla u}\id x+2\abss{D^s u}(\Omega).
    \end{align*} The last line converges to $2\abss{D^s u}$ as $M \to \infty$, which gives the desired result.
\end{proof}
\par The following lemma shows that area-strict convergence is retained on good subsets.
\begin{lemma}\label{lem:strict and area}
		Let $\Om\subset \R^n$ be open. Suppose $u\in \dot{\BV}(\Om)$ and $\{u_{\lambda}\}_{\lambda>0} \in \dot{\BV}(\Om)$, and that $u_{\lambda}\to u$ area-strictly
		in $\Om$.
        Suppose $W\subset \Om$ is open with
        $|Du|(\partial W\cap \Om)=0=\Le^n(\partial W \cap\Om)$.
        Then $u_{\lambda} \to u$ area-strictly in $W$ as well, and thus
        \begin{align}
        &\lim_{\lambda\to \infty} \int_W E(Du_{\lambda})
	=\int_{W} E(Du). \label{eq:strict conv consequence three}
    \end{align}
    If $W$ is moreover bounded, we also have
    \begin{align}
		&\lim_{\lambda\to \infty}Du_\lambda(W)=Du(W), \nonumber\\
		&\lim_{\lambda\to \infty}|Du_\lambda|(W)=|Du|(W).\label{eq:strict conv consequence two}
		\end{align}
	\end{lemma}
   %The boundedness assumption could be removed, but we will not need this.
\begin{proof}
We have
\[
\lim_{\lambda\to \infty} \int_\Om E(Du_{\lambda})=\int_{\Om} E(Du),
\]
and by the lower semicontinuity of the area functional with respect to $L^1_{\loc}$-convergence,
see \cite{ADM, FM},
we have
\[
\liminf_{\lambda\to \infty} \int_W E(Du_{\lambda})\ge \int_{W} E(Du)
\quad\textrm{and}\quad
\liminf_{\lambda\to \infty} \int_{\Om\setminus \overline{W}} E(Du_{\lambda})
\ge \int_{\Om\setminus \overline{W}} E(Du).
\]
Since $\int_{\partial W\cap \Om}E(Du)=0$, these are necessarily equalities.
In particular, \eqref{eq:strict conv consequence three} follows.
%This follows easily from Lemma \ref{lem:areastrYM}. More precisely,
If $W$ is bounded, the proof is completed by noting that 
    \[ 
    \lim_{\lambda\to\infty} \int_{W}  g(D u_{\lambda}) 
    = \int_{W} 
    \brangle{\nu_x, g(\cdot)}\id x 
    + \int_{W} \brangle{\nu_x^{\infty},g^{\infty}(\cdot)}\id \lambda_{\nu} 
    = \int_W g(Du)  
    \] 
%    \[ 
%    \lim_{\lambda_{\to}\infty} \int_{\Omega} \mathbbm{1}_W g(D u_{\lambda}) = \int_{\Omega} \mathbbm{1}_W \brangle{\nu_x, g(\cdot)}\id x + \int_{\Omega}\mathbbm{1}_W \brangle{\nu_x^{\infty},g^{\infty}(\cdot)}\id \lambda_{\nu} = \int_W g(Du)  
%    \] 
    by Lemma \ref{lem:areastrYM},
    %and Proposition \ref{prop:YM Caratheodory},
    where  $\nu = ((\nu_x)_{x\in W}, \lambda_{\nu}, (\nu_x^{\infty})_{x\in \overline{W}})$ is the Young measure generated by $\{u_{\lambda}\}$ on $W$, and $g(z)$ is taken to be $z, \abss{z}$,
    %$z, \abss{z}, E(z)$,
    respectively.
\end{proof}

\subsection{Lower bound in one dimension}

We will first prove the lower bound for the $\Gamma$-convergence in one dimension.
Recall that we always consider $0<\gamma<\infty$.
When considering families $\{u_\lambda\}_{\lambda>0}$,
we will rely on certain estimates known for a fixed function $u$.
In particular, the following result will be needed; it was originally proved in \cite{Pic}.

\begin{proposition}[{\cite[Proposition 3.18]{L-sharp}}]\label{thm:1d case variation abs cont pointwise}
	Let $\Om\subset \R$ be open and let $u\in \dot{\BV}(\Om)$.
	Then
	\[
	\liminf_{\lambda \to \infty}F_{\gamma,\lambda}\left(u,\Om\right)
	\ge \frac{2}{\gamma}|D^a u|(\Om).
	\]
\end{proposition}

By \cite[Theorem 1.4]{BSSY} we know that for an open interval $\Om\subset \R$
and $u\in L^1_{\loc}(\Om)$, we have
\begin{equation}\label{eq:sup bound}
	F_{\gamma,\lambda}\left(u,\Om\right)
	\le C'\Var(u,\Om)
\end{equation}
for all $\lambda>0$ and for some constant $C'$ depending only on $n$ and $\gamma$.
In fact, this is proved in \cite{BSSY} only in the case $\Om=\R^n$, but when $n=1$
it clearly holds for any open interval $\Om$ as well, since we find an extension $v$ of $u$
with $|Dv|(\R)=|Du|(\Om)$, and then
\begin{align*}
    F_{\gamma,\lambda}(u,\Om)
    \le F_{\gamma,\lambda}(v,\R)
	\le C' \Var(v,\R)
    =C' \Var(u,\Om).
\end{align*}

Now we prove the following generalization of Proposition \ref{thm:1d case variation abs cont pointwise}.

\begin{proposition}\label{prop:1d case variation abs cont}
	Let $\Om\subset \R$ be open, let 
	$u\in \dot{\BV}(\Om)$,
	and let $u_{\lambda}\to u$ area-strictly in $\Om$ as $\lambda\to\infty$.
	Then
	\[
	\liminf_{\lambda \to \infty}F_{\gamma,\lambda}\left(u_\lambda,\Om\right)
	\ge \frac{2}{\gamma}|D^a u|(\Om).
	\]
\end{proposition}

\begin{proof}
	Fix $\eps>0$.
	%We also find an open set $W\Subset W'$ such that $|Df|(W\setminus W')<\eps$.
	First consider a bounded open interval $W\subset \Om$ with $|Du|(\partial W\cap \Om)=0$.
    Then by Lemma \ref{lem:strict and area},
    we also have $u_{\lambda}\to u$ area-strictly in $W$ as $\lambda\to\infty$.
	Note that, for any $\varepsilon>0$,
	\begin{align*}
		&\left\{(x,y)\in W\times W\colon
		\frac{|u(x)-u(y)|}{|x-y|^{1+\gamma}}>(1+\eps)\lambda\right\}\\
		&\quad\subset  \left\{(x,y)\in W\times W\colon
		\frac{|u_\lambda(x)-u_\lambda(y)|}{|x-y|^{1+\gamma}}>\lambda\right\}\\
		&\qquad \cup  \left\{(x,y)\in W\times W\colon
		\frac{|(u-u_\lambda)(x)-(u-u_\lambda)(y)|}{|x-y|^{1+\gamma}}
		>\eps\lambda\right\}.
	\end{align*}
	Thus, we have
	\[
	\frac{1}{1+\eps}F_{\gamma,(1+\eps)\lambda}\left(u,W\right)
	\le F_{\gamma,\lambda}\left(u_\lambda,W\right)
	+\eps^{-1}F_{\gamma,\eps\lambda}\left(u-u_\lambda,W\right).
	\]
	Note that $F_{\gamma,\lambda}(v,W)<\infty$ whenever $v\in \dot{\BV}(\Om)$
	by \eqref{eq:sup bound}, and so we can subtract
	\begin{equation}\label{eq:interval case}
	\begin{split}
		&\liminf_{\lambda \to \infty}F_{\gamma,\lambda}\left(u_\lambda,W\right)\\
		&\qquad \ge \frac{1}{1+\eps}\liminf_{\lambda \to \infty}F_{\gamma,(1+\eps)\lambda}\left(u,W\right)
		-\eps^{-1}\limsup_{\lambda \to \infty}F_{\gamma,\eps\lambda}\left(u-u_\lambda,W\right)\\
		&\qquad\ge \frac{1}{1+\eps}\liminf_{\lambda \to \infty}F_{\gamma,(1+\eps)\lambda}\left(u,W\right)
		-C'\eps^{-1}\limsup_{\lambda \to \infty}|D(u-u_\lambda)|(W)\quad\textrm{by }
		\eqref{eq:sup bound}\\
		&\qquad\ge \frac{1}{1+\eps}\liminf_{\lambda \to \infty}F_{\gamma,(1+\eps)\lambda}\left(u,W\right)
		-2C'\eps^{-1} |D^s u|(W)\quad\textrm{by Proposition }
		\ref{prop:area strict and norm}\\
		&\qquad\ge \frac{1}{1+\eps} \frac{2}{\gamma}|D^au|(W) -2C'\eps^{-1} |D^s u|(W)
		\quad\textrm{by Proposition }\ref{thm:1d case variation abs cont pointwise}.
	\end{split}
	\end{equation}
	Then consider an arbitrary bounded open set $W\subset \Om$ with $|Du|(\partial W\cap \Om)=0$.
	We can write it as a union of pairwise disjoint open intervals $W=\bigcup_{j=1}^{\infty}W_j$
    with $|Du|(\partial W_j\cap \Om)=0$ for each $j\in\N$.
	Fix $N\in\N$. Now
	\begin{align*}
		\liminf_{\lambda \to \infty}F_{\gamma,\lambda}\left(u_\lambda,W\right)
		&\ge \liminf_{\lambda \to \infty}F_{\gamma,\lambda}\left(u_\lambda,\bigcup_{j=1}^N W_j\right)\\
		&\ge \sum_{j=1}^N\liminf_{\lambda \to \infty}F_{\gamma,\lambda}\left(u_\lambda, W_j\right)\\
		&\ge \sum_{j=1}^N\left[
		 \frac{1}{1+\eps} \frac{2}{\gamma}|D^au|(W_j) -2C'\eps^{-1} |D^s u|(W_j)\right]\\
		&\ge \frac{1}{1+\eps} \frac{2}{\gamma}|D^au|
        \left(\bigcup_{j=1}^N W_j\right)
		 -2C'\eps^{-1} |D^s u|(W).
	\end{align*}
	Letting $N\to \infty$, we get
	\[
	\liminf_{\lambda \to \infty}F_{\gamma,\lambda}(u_\lambda,W)
	\ge \frac{1}{1+\eps} \frac{2}{\gamma}|D^au|(W)
	-2C'\eps^{-1} |D^s u|(W).
	\]
	We can choose the bounded
    open set $W\subset \Om$ such that $|Du|(\partial W\cap \Om)=0$,
    $|D^s u|(W)<\eps^2$, and $|D^a u|(\Om\setminus W)<\eps$.
	It follows that
	\[
        \liminf_{\lambda \to \infty}F_{\gamma,\lambda}(u_\lambda,\Om)
		\ge \liminf_{\lambda \to \infty}F_{\gamma,\lambda}(u_\lambda,W)
		\ge \frac{1}{1+\eps} \frac{2}{\gamma}\left(|D^au|(\Om)-\eps\right)-2C'\eps.
	\]
	Letting $\eps\to 0$, we get the result.
\end{proof}

We also need the following key estimate, which is essentially the main result of
\cite{L-sharp}.
In \cite[Corollary 3.9]{L-sharp}, this was formulated in terms of a limit
$\liminf_{\lambda\to\infty}$, but the following version follows immediately from the proof there.

	\begin{proposition}\label{cor:1d case preli}
	Let $-\infty<b_1\le b_2<\infty$ and $-\infty<L_1< L_2<\infty$, and let 
	$u$ be increasing with $u(x)=b_1$ for $x< L_1$ and $u(x)=b_2$ for $x> L_2$.
	Let $\gamma>0$.
	Then for every $\delta>0$ and $\lambda>(b_2-b_1)/\delta^{1+\gamma}$, we have
	\[
	\lambda \nu_\gamma(\left\{(x,y)\in (L_1-\delta,L_2+\delta)^2\colon
	|u(x)-u(y)|>\lambda|x-y|^{1+\gamma}\right\})
	\ge \frac{2(b_2-b_1)}{\gamma+1}.
	\]
\end{proposition}

\begin{proposition}\label{prop:1d case variation singular abs cont}
	Let $\Om\subset \R$ be open and let 
	$u\in \dot{\BV}(\Om)$,
	and let $u_{\lambda}\to u$ as $\lambda\to\infty$ area-strictly in $\Om$.
	Then
	\[
	\liminf_{\lambda \to \infty}F_{\gamma,\lambda}\left(u_\lambda,\Om\right)
	\ge \frac{2}{\gamma+1}|D u|(\Om).
	\]
\end{proposition}
In the proof, we will denote by $(Du)_+$ and $(Du)_-$ the positive and negative parts of $Du$,
    so that $Du=(Du)_+-(Du)_-$.

\begin{proof}
	First assume that $\Om$
    is a bounded interval
    %that is,
    $\Om=(a,b)$.
    %for some $-\infty\le a<b\le \infty$.
	We can assume that $u$ is the pointwise representative
	\begin{equation}\label{eq:pw repr}
		u(x)=\limsup_{r\to 0}\,\vint{B(x,r)}u\id \Le^1
		\quad\textrm{for all }x\in \Om.
	\end{equation}
	By BV theory on the real line, see \cite[Section 3.2]{AFP},
	we know that $\lim_{x\to a^+}u(x)$ exists, and that for all Lebesgue points $y\in \Omega$ of $u$,
	\[
	u(y)=\lim_{x\to a^+}u(x)+Du((a,y]).
	\]
    We can assume that $u_{\lambda}\in \dot{\BV}(\Om)$ for all $\lambda>0$, and then the analogous
    fact holds for each $u_{\lambda}$.
	Fix $0<\eps<1$. For $|Du|$-a.e. $x\in \Om$, we have that
	\[
	\frac{|Du([x-r,x+r])|}{|Du|([x-r,x+r])}>1-\eps^2
	\]
    if $r>0$ is small enough.
    %with $r>0$ small enough.
	Moreover, for all except at most countably many $r>0$, we have
	$|Du|(\{x-r\})=0=|Du|(\{x+r\})$.
	Thus by Vitali's covering theorem,
	we find a sequence of pairwise disjoint compact intervals $[c_j,d_j]\subset \Om$, $j\in\N$, such that
	\begin{equation}\label{eq:cj and dj eps}
		\frac{|Du([c_j,d_j])|}{|Du|([c_j,d_j])}>1-\eps^2,
	\end{equation}
	$|Du|\left(\Om\setminus \bigcup_{j=1}^{\infty}[c_j,d_j]\right)=0$,
	and $|Du|(\{c_j\})=|Du|(\{d_j\})=0$.

	For sufficiently large $N\in\N$, we have
	\begin{equation}\label{eq:j one to N}
		|Du|\left(\Om\setminus \bigcup_{j=1}^{N}[c_j,d_j]\right)<\eps^2.
	\end{equation}
	We find $\delta>0$ such that the intervals $[c_j-\delta,d_j+\delta]$, $j=1,\ldots,N$,
	are still pairwise disjoint
	and contained in $\Om$.
	%We also have $f_{\lambda}\to f$ strictly, see e.g. \cite[Theorem 5]{KrRind}.
	By Lemma \ref{lem:strict and area}, we have
    
    \begin{equation}\label{eq:strict conv in Omega}
    \lim_{\lambda\to\infty}|Du_\lambda|(\Om)=|Du|(\Om),
    \end{equation}
    as well as
	\[
	\lim_{\lambda\to\infty}Du_\lambda((c_j,d_j))\to Du([c_j,d_j])
		\quad\textrm{and}\quad
	\lim_{\lambda\to\infty}|Du_\lambda|((c_j,d_j))=|Du|([c_j,d_j])
	\]
	for all $j=1,\ldots,N$.
	Thus, for sufficiently large $\lambda$ there hold, for $j =1,\dots,N$,  
	\begin{align}
		&\frac{|Du_\lambda((c_j,d_j))|}{|Du_\lambda|((c_j,d_j))}>1-\eps^2 \quad \mbox{by }\eqref{eq:cj and dj eps} \label{eq:cj and dj lambda},\\
        &\abs{\abss{Du_{\lambda}}((c_j,d_j))-\abss{Du}([c_j,d_j])} < \varepsilon^2/N, \label{eq:j strict eps}\\
        \mbox{and} \quad &|Du_{\lambda}|\left(\Om\setminus \bigcup_{j=1}^{N}(c_j,d_j)\right)<\eps^2 \quad \mbox{by }\eqref{eq:j one to N} \label{eq:j one to N lambda}.
	\end{align}
	Consider such sufficiently large $\lambda$ and define a measure $\mu_\lambda$ on $\Om$ as follows:
	\[
	\begin{cases}
		\mu_\lambda:=(Du_\lambda)_+\quad\textrm{on }[c_j,d_j]\quad\textrm{if }Du_\lambda([c_j,d_j])>0\\
		\mu_\lambda:=-(Du_\lambda)_-\quad\textrm{on }[c_j,d_j]\quad\textrm{if }Du_\lambda([c_j,d_j])<0,
	\end{cases}
	\]
	for $j=1,\dots,N$, and $\mu_\lambda:=0$ outside these intervals.
	From \eqref{eq:cj and dj lambda}, we have
	\begin{equation}\label{eq:epsilon condition}
		|Du_\lambda-\mu_\lambda|((c_j,d_j))<\eps^2 |Du_\lambda|((c_j,d_j)).
	\end{equation}
	Define
	\[
	v_\lambda(y):=\lim_{x\to a^+}u_\lambda(x)+\mu_\lambda((a,y])\quad \textrm{for }y\in \Om.
	\]
    Then
	\begin{equation}\label{eq:D fg difference}
		\begin{split}
			|D(u_{\lambda}-v_{\lambda})|(\Om)
			&=|Du_{\lambda}-\mu_{\lambda}|(\Om)\\
			&= \sum_{j=1}^N |Du_{\lambda}-\mu_{\lambda}|((c_j,d_j))
			+|Du_{\lambda}|\Bigg(\Om\setminus \bigcup_{j=1}^N (c_j,d_j)\Bigg)\\
            & \leq \sum_{j=1}^N \eps^2|Du_{\lambda}|((c_j,d_j)) + \eps^2 \quad \mbox{by }\eqref{eq:epsilon condition}\, , \,\eqref{eq:j one to N lambda}\\
			&\le \eps^2(|Du|(\Om)+\eps^2)+\eps^2\quad\textrm{by }\eqref{eq:j strict eps}.
		\end{split}
	\end{equation}
    
	Note that
	\begin{equation}\label{eq:pV estimate}
		\begin{split}
			\sum_{j=1}^N|v_{\lambda}(c_j-\delta)-v_{\lambda}(d_j+\delta)|
			&=\sum_{j=1}^N|Dv_{\lambda}((c_j,d_j))|\\
			%&= \sum_{j=1}^N|Dv_{\lambda}|([c_j,d_j])\\
			&\ge (1-\eps^2)\sum_{j=1}^N|Du_{\lambda}|((c_j,d_j))\quad\textrm{by }\eqref{eq:epsilon condition}\\
            &\ge (1-\eps^2) \sum_{j=1}^N (\abss{Du}([c_j,d_j]) -\eps^2/N) \quad \mbox{by }\eqref{eq:j strict eps} \\
			&\ge (1-\eps^2)(|Du|(\Om)-2\eps^2)\quad\textrm{by }\eqref{eq:j one to N}.
		\end{split}
	\end{equation}
	It is also easy to see that
	\begin{align*}
		&\left\{(x,y)\in \Om\times \Om\colon
		\frac{|v_\lambda(x)-v_\lambda(y)|}{|x-y|^{1+\gamma}}>(1+\eps)\lambda\right\}\\
		&\quad\subset  \left\{(x,y)\in \Om\times \Om\colon
		\frac{|u_\lambda(x)-u_\lambda(y)|}{|x-y|^{1+\gamma}}>\lambda\right\}\\
		&\qquad \cup  \left\{(x,y)\in \Om\times \Om\colon
		\frac{|(v_\lambda-u_\lambda)(x)-(v_\lambda-u_\lambda)(y)|}{|x-y|^{1+\gamma}}
		>\eps\lambda\right\}.
	\end{align*}
    Using the fact that $\Om$ is a bounded interval, as well as
    \eqref{eq:strict conv in Omega} and the definition of $\{v_{\lambda}\}$, for some sufficiently large $\Lambda>0$ we have
    \[
    \sup_{\lambda>\Lambda} \osc_{\Omega} v_{\lambda}
    \le \sup_{\lambda>\Lambda}|Du_{\lambda}|(\Om)<\infty,
    \]
    where $\osc_{\Omega}v := \esssup \{\abss{v(x)-v(y)}\col x,y \in \Omega\}$ is the oscillation of the function $v$ on $\Omega$. Now for
    $\lambda>\max\{\Lambda,\sup_{\lambda'>\Lambda} \osc_{\Omega} v_{\lambda}/\delta^{1+\gamma}\}$,
    Proposition \ref{cor:1d case preli} becomes available for use below.
	Then, we have the following estimate
	\begin{align*}
		&\liminf_{\lambda \to \infty}F_{\gamma,\lambda}\left(u_{\lambda},\Om\right)\\
		&\qquad \ge \frac{1}{1+\eps}\liminf_{\lambda \to \infty}F_{\gamma,(1+\eps)\lambda}\left(v_{\lambda},\Om\right)
		-\eps^{-1}\limsup_{\lambda \to \infty}F_{\gamma,\eps\lambda}\left(v_{\lambda}-u_{\lambda},\Om\right)\\
		&\qquad\ge \frac{1}{1+\eps}\liminf_{\lambda \to \infty}F_{\gamma,(1+\eps)\lambda}\left(v_{\lambda},\Om\right)
		-C'\eps^{-1}|D(u_{\lambda}-v_{\lambda})|(\Om)\quad\textrm{by }\eqref{eq:sup bound}\\
		&\qquad\ge \frac{1}{1+\eps}\sum_{j=1}^N\liminf_{\lambda \to \infty}F_{\gamma,(1+\eps)\lambda}\left(v_{\lambda},(c_j-\delta,d_j+\delta)\right)
		-C'\eps^{-1} \eps^2 (|Du|(\Om)+2)\quad\textrm{by }\eqref{eq:D fg difference}\\
		&\qquad\ge \frac{1}{1+\eps}\sum_{j=1}^N\frac{2}{\gamma+1}\liminf_{\lambda \to \infty}
        |v_{\lambda}(c_j-\delta)-v_{\lambda}(d_j+\delta)|
		-C'\eps (|Du|(\Om)+2)
		\quad\textrm{by Prop. }\ref{cor:1d case preli}\\
		&\qquad\ge \frac{1}{1+\eps}\frac{2}{\gamma+1}(1-\eps^2)(|Du|(\Om)-2\eps^2)
		-C'\eps (|Du|(\Om)+2)
	\end{align*}
	by \eqref{eq:pV estimate}.
	Letting $\eps\to 0$, we get the result.
	
	For a general bounded open $\Om\subset \R$, we have
    $\Om=\bigcup_{k=1}^{\infty}\Om_k$, where the $\Om_k$'s
	are disjoint bounded open intervals, possibly empty.
	For $M\in\N$, we have 
	\begin{align*}
		\liminf_{\lambda \to \infty}F_{\gamma,\lambda}\left(u_{\lambda},\Om\right)
		&\ge \sum_{k=1}^M\liminf_{\lambda \to \infty}F_{\gamma,\lambda}\left(u_{\lambda},\Om_k\right)\\
		&\ge \frac{2}{\gamma+1}\sum_{k=1}^M|Du|(\Om_k).
	\end{align*}
	Letting $M\to\infty$, we get
	\[
	\liminf_{\lambda \to \infty}F_{\gamma,\lambda}\left(u_{\lambda},\Om\right)
	\ge \frac{2}{\gamma+1}|Du|(\Om).
	\]
    Finally, for a general open $\Om\subset \R$, one easily gets the result by
    considering $\Om\cap B(0,R)$ for $R\to \infty$.
\end{proof}

Now we get the lower bound for the
$\Gamma_{AS}$-limit
in the one-dimensional case.

\begin{proposition}\label{prop:1d case variation}
	Let $\Om\subset \R$ be open, let $u\in \dot{\BV}(\Om)$,
	and let $u_{\lambda}\to u$ area-strictly in $\Om$.
	Then
	\[
	\liminf_{\lambda \to \infty}F_{\gamma,\lambda}\left(u_\lambda,\Om\right)
	\ge \frac{2}{\gamma}|D^a u|(\Om)+\frac{2}{\gamma+1}|D^s u|(\Om).
	\]
\end{proposition}
\begin{proof}
	Fix $\eps>0$.
We find an open set $W'\subset \Om$ such that $|D^a u|(W')<\eps$ and $|D^s u|(\Om\setminus W')=0$.
We also find an open set $W\Subset W'$ such that $|Du|(W'\setminus W)<\eps$
and $|Du|(\partial W)=0$.
Now we have $u_{\lambda}\to u$ area-strictly in $W$ and in $\Om\setminus \overline{W}$
by \eqref{eq:strict conv consequence three}. Thus
\begin{align*}
	\liminf_{\lambda \to \infty}F_{\gamma,\lambda}\left(u,\Om\right)
	&\ge \liminf_{\lambda \to \infty}F_{\gamma,\lambda}\left(u,W\right)
	+\liminf_{\lambda \to \infty}F_{\gamma,\lambda}\left(u,\Om\setminus \overline{W}\right)\\
	&\ge \frac{2}{\gamma+1}|D u|(W)+\frac{2}{\gamma}|D^a u|(\Om\setminus \overline{W})
	\quad\textrm{by Propositions }\ref{prop:1d case variation singular abs cont},
	\ref{prop:1d case variation abs cont}\\
	&\ge \frac{2}{\gamma+1}(|D^s u|(\Om) -\eps)
	+ \frac{2}{\gamma}(|D^a u|(\Om)-\eps).
\end{align*}
Letting $\eps\to 0$, we get the result.
\end{proof}
	
		\subsection{Higher dimensions}

        Now we can generalize from one dimension to higher dimensions with standard techniques.
        As before, we denote the unit sphere in $\R^n$ by $\mathbb S^{n-1}$.
        
			\begin{lemma}\label{lem:area strict in directions}
			Suppose $\Om\subset \R^n$ is open and bounded,
            $u\in \dot{\BV}(\Om)$, and $u_\lambda\to u$ area-strictly in $\Om$.
			Then for all $\sigma\in \mathbb S^{n-1}$,
			\[
			\lim_{\lambda\to\infty}\left(\int_{\Om}\sqrt{1+|\nabla u_\lambda\cdot \sigma|^2}\id x
            +|D^s u_\lambda\cdot \sigma|(\Om)\right)
			=\int_{\Om}\sqrt{1+|\nabla u\cdot \sigma|^2}\id x+|D^s u\cdot \sigma|(\Om).
			\]
		\end{lemma}
		
		\begin{proof}
			Fix $\sigma\in \mathbb S^{n-1}$.
            Define $h(v):=\sqrt{1+|v\cdot \sigma|^2}$, so that $h \in \EB(\R^n)$
            (recall Section \ref{subsec: Young}) and
            $h^{\infty}(v)=|v\cdot \sigma|$.
			Then this result follows from Lemma \ref{lem:areastrYM}.
		\end{proof}
		
		For a unit vector $\sigma\in \mathbb S^{n-1}$, denote by $\langle \sigma\rangle^\perp$ the $(n-1)$-dimensional hyperplane perpendicular to $\sigma$ and containing the origin.
		For an open set $\Om\subset \R^n$ and $z\in\langle \sigma\rangle^\perp$, also denote
		$\Om_{\sigma,z}:=\{t\in \R\colon z+\sigma t\in \Om\}$ and
		\[
		u_{\sigma,z}(t):=u(z+\sigma t),\quad  t\in \Om_{\sigma,z}.
		\]
        For $u\in \dot{\BV}(\Om)$, we note the following facts that can be found in \cite[Section 3.11]{AFP}:
		for $\sigma\in \mathbb S^{n-1}$,
        for a.e. $z\in \langle \sigma\rangle^\perp$ and a.e. $t\in \Om_{\sigma,z}$
        we have
		\begin{equation}\label{eq:one dim Df}
            \nabla u(z+\sigma t) \cdot \sigma = u'_{\sigma,z}(t),
		\end{equation}
        where $u'_{\sigma,z}$ is the density of $D^a u_{\sigma,z}$,
		and then
		\begin{align}
            \int_{\Om}\id |D^a u\cdot \sigma|
			&=\int_{\langle \sigma\rangle^\perp}|D^a u_{\sigma,z}|(\Om_{\sigma,z})\id z,\label{eq: one dim Df int}\\
			C_n\int_{\Om}\id |D^a u|
			&=\int_{\mathbb S^{n-1}}\int_{\langle \sigma\rangle^\perp}|D^a u_{\sigma,z}|(\Om_{\sigma,z}) \label{eq:change of var for Df}
            \id z\id \sigma.
		\end{align}
		The analogous fact with the absolutely continuous parts replaced by the
        Cantor or jump part also holds true.

\begin{lemma}\label{lem:lim inf}
Let $h,\{h_j\}_{j=1}^{\infty}$ be nonnegative functions in $L^1(\R^{n-1})$, such that
$\liminf_{j\to\infty}h_j(y)\ge h(y)$ for $\mathcal L^{n-1}$-a.e. $y\in \R^{n-1}$, and
\[
\lim_{j\to\infty}\int_{\R^{n-1}}h_j\,\id\mathcal L^{n-1}=\int_{\R^{n-1}}h\,\id\mathcal L^{n-1}.
\]
Then for a subsequence (not relabeled) we have $h_j(y)\to h(y)$ for $\mathcal L^{n-1}$-a.e. $y\in \R^{n-1}$.
\end{lemma}
\begin{proof}
From the assumption, we know $(h-h_j)_+\leq h$ pointwise. Then by the dominated convergence theorem we get
\[
\lim_{j\to\infty}\int_{\R^{n-1}}(h-h_j)_+\,\id\mathcal L^{n-1}=0.
\]
On the other hand,
\begin{align*}
&\limsup_{j\to\infty}\int_{\R^{n-1}}(h_j-h)_+\,\id\mathcal L^{n-1}\\
&\quad \le \limsup_{j\to\infty}\int_{\R^{n-1}}(h_j-h)\,\id\mathcal L^{n-1}
+\limsup_{j\to\infty}\int_{\R^{n-1}}(h-h_j)_+\,\id\mathcal L^{n-1}\\
&\quad =0.
\end{align*}
In conclusion, $h_j\to h$ in $L^1(\R^{n-1})$ and thus we can find the pointwise converging subsequence.
\end{proof}

		\begin{proposition}\label{prop:area strict on slices}
		Suppose $\Om\subset \R^n$ is open, $u\in \dot{\BV}(\Om)$, and $u_\lambda\to u$ area-strictly in $\Om$.
		Let $\sigma\in \mathbb S^{n-1}$. Then for some subsequence $\lambda_j\to \infty$
        we have $(u_{\lambda_j})_{\sigma,z}\to u_{\sigma,z}$ area-strictly for a.e. $z\in \langle \sigma\rangle^\perp$.
		\end{proposition}
		\begin{proof}
		
		By Fubini's theorem, for a subsequence $\lambda_j\to \infty$ we get
		$(u_{\lambda_j})_{\sigma,z}\to u_{\sigma,z}$ in $L^1_{\loc}(\Om_{\sigma,z})$
		for $\Le^{n-1}$-a.e. $z\in \langle \sigma\rangle^\perp$.
		The area functional is lower semicontinuous with respect to $L^1_{\loc}$-convergence,
        and so for a.e. $z\in \langle \sigma\rangle^\perp$ we get
		\begin{equation}\label{eq:area functional and lsc}
		\begin{split}
		&\int_{\Om_{\sigma,z}}\sqrt{1+(u'_{\sigma,z}(t))^2}\id t+|D^s u_{\sigma,z}|(\Om_{\sigma,z})\\
		&\qquad \le\liminf_{j\to\infty}\left(\int_{\Om_{\sigma,z}}\sqrt{1+((u_{\lambda_j})'_{\sigma,z}(t))^2}\id t
		+|D^s (u_{\lambda_j})_{\sigma,z}|(\Om_{\sigma,z})\right).
		\end{split}
		\end{equation}
		On the other hand, from \eqref{eq:one dim Df} and Lemma \ref{lem:area strict in directions} we get
		\begin{equation}\label{eq:area functional and continuity}
		\begin{split}
		\lim_{j\to\infty}\int_{\langle \sigma\rangle^\perp}\left(\int_{\Om_{\sigma,z}}
		\sqrt{1+((u_{\lambda_j})'_{\sigma,z}(t))^2}\id t
		+|D^s (u_{\lambda_j})_{\sigma,z}|(\Om_{\sigma,z})\right)\id z\\
		=\int_{\langle \sigma\rangle^\perp}\left(\int_{\Om_{\sigma,z}}
		\sqrt{1+(u'_{\sigma,z}(t))^2}\id t
		+|D^s u_{\sigma,z}|(\Om_{\sigma,z})\right)\id z.
		\end{split}
		\end{equation}
		Combining \eqref{eq:area functional and lsc} with
		\eqref{eq:area functional and continuity}, by Lemma \ref{lem:lim inf}
		it follows that for a further subsequence (not relabeled), for
        a.e. $z\in \langle \sigma\rangle^\perp$ we have
		\begin{align*}
				&\int_{\Om_{\sigma,z}}\sqrt{1+(u'_{\sigma,z}(t))^2}\id t+|D^s u_{\sigma,z}|(\Om_{\sigma,z})\\
				&\qquad =\lim_{j\to\infty}\left(\int_{\Om_{\sigma,z}}\sqrt{1+((u_{\lambda_j})'_{\sigma,z}(t))^2}\id t
				+|D^s (u_{\lambda_j})_{\sigma,z}|(\Om_{\sigma,z})\right).
		\end{align*}
	\end{proof}

    Now we prove the lower bound for the
    $\Gamma_{AS}$-limit.
    
	\begin{theorem}\label{thm:Gamma area strict lower bound}
			Suppose $\Om\subset \R^n$ is open, $u\in L^1_{\loc}(\Om)$, and $u_\lambda\to u$ area-strictly in $\Om$. Then
		\begin{equation}\label{eq:BV lower limit}
			\liminf_{\lambda\to \infty}F_{\gamma,\lambda}(u_{\lambda},\Om)
			\ge \frac{C_n}{\gamma}|D^a u|(\Om)
			+ \frac{C_n}{\gamma+1}|D^s u|(\Om).
		\end{equation}
	\end{theorem}
	
	\begin{proof}
    We can assume that $\liminf_{\lambda\to \infty}F_{\gamma,\lambda}(u_{\lambda},\Om)<\infty$.
    Then by Gobbino--Picenni \cite[Theorem 1.1]{MaGo}, we know that $|Du|(\Om)<\infty$.
    Then by \eqref{eq: one dim Df int}, \eqref{eq:change of var for Df} and the analogues for the jump and Cantor parts,
    we have that for any $\sigma\in \mathbb S^{n-1}$, $u_{\sigma,z}\in \dot{\BV}(\Om)$ for a.e. $z\in \langle \sigma\rangle^\perp$,
    which makes Proposition
	\ref{prop:1d case variation} available for use below.
    Since $u_\lambda\to u$ area-strictly in $\Om$, we have $u_\lambda\in \dot{\BV}(\Om)$
    for all sufficiently large $\lambda$, and we can assume that this is in fact true for all $\lambda>0$.
    
		Antonucci et al. \cite[p. 622]{AGMP} note that 
		for $g\in L^1(\Om\times\Om)$, the following change-of-variable formula holds true
		\[
		\int_{\R^n}\int_{\R^n} g(x',y')\id x'\id y'
		=\frac 12 \int_{\mathbb S^{n-1}} \int_{\langle \sigma\rangle^\perp}
		\int_{\R}\int_{\R}g(z+\sigma x,z+\sigma y)|x-y|^{n-1}\id x\id y\id z\id \sigma.
		\]
		Applying it to $g \mathbbm{1}_{\Omega \times \Omega}$, we have
		\begin{equation}\label{eq:Om multiple integral}
			\begin{split}
				\int_{\Om}\int_{\Om} g(x',y')\id x'\id y'
				=\frac 12 \int_{\mathbb S^{n-1}} \int_{\langle \sigma\rangle^\perp}
				\int_{\Om_{\sigma,z}}\int_{\Om_{\sigma,z}}g(z+\sigma x,z+\sigma y)|x-y|^{n-1}
                \id x\id y\id z\id \sigma.
			\end{split}
		\end{equation}
		Then the functional $F_{\gamma,\lambda}(u_{\lambda})$ can be written as
		\begin{align*}
			&\lambda \nu_\gamma\left(\left\{(x',y')\in \Om\times \Om\colon
			\frac{|u_{\lambda}(x')-u_{\lambda}(y')|}{|x'-y'|^{1+\gamma}}>\lambda\right\}\right)\\
			& =\lambda\int_{\Om}\int_{\Om} \mathbbm{1}_{\{(x',y')\in \Om\times \Om\colon
				|u_{\lambda}(x')-u_{\lambda}(y')|/|x'-y'|^{1+\gamma}>\lambda\}}(x',y')|x'-y'|^{\gamma-n}\id x'\id y'\\
			&= \frac {\lambda}{2}\int_{\mathbb S^{n-1}}\int_{\langle \sigma\rangle^\perp} 
			\int_{\Om_{\sigma,z}}\int_{\Om_{\sigma,z}}\\
            & \quad \mathbbm{1}_{\{(x,y)\in \Om_{\sigma,z}
				\times \Om_{\sigma,z}\colon
				|u_{\lambda}(z+\sigma x)-u_{\lambda}(z+\sigma y)|/|x-y|^{1+\gamma}>\lambda\}}(x,y)
			 \cdot |x-y|^{\gamma-1}\id x\id y\id z\id \sigma \\
            &= \frac {1}{2}\int_{\mathbb S^{n-1}}\int_{\langle \sigma\rangle^\perp} F_{\gamma, \lambda}((u_{\lambda})_{\sigma,z},\Omega_{\sigma,z}) \id z \id \sigma ,
        \end{align*} where the last line follows from \eqref{eq:Om multiple integral}.
        Take the sequence $\lambda_j\to \infty$
        given by Proposition \ref{prop:area strict on slices}.
        Then by Fatou's lemma we have
        \begin{align*}
            &\liminf_{j\to\infty}\lambda_j \nu_\gamma\left(\left\{(x',y')\in \Om\times \Om\colon
			\frac{|u_{\lambda_j}(x')-u_{\lambda_j}(y')|}{|x'-y'|^{1+\gamma}}>\lambda_j\right\}\right)\\
			&\ge \frac {1}{2}\int_{\mathbb S^{n-1}}\int_{\langle \sigma\rangle^\perp}\liminf_{j\to\infty} F_{\gamma,\lambda_j}((u_{\lambda_j})_{\sigma,z},\Omega_{\sigma,z}) \id z\id \sigma\\
			& \ge \int_{\mathbb S^{n-1}}\int_{\langle \sigma\rangle^\perp}
			\left[\frac{1}{\gamma}|D^a u_{\sigma,z}|(\Om_{\sigma,z})\,dz\,d\sigma
			+\frac{1}{\gamma+1}
			|D^s u_{\sigma,z}|(\Om_{\sigma,z})\right]\id z\id \sigma
			\quad\textrm{by Prop. }\ref{prop:area strict on slices},
			\ref{prop:1d case variation}\\
			& = \frac{C_n}{\gamma}|D^a u|(\Om)+\frac{C_n}{\gamma+1}|D^s u|(\Om) \quad \textrm{by \eqref{eq:change of var for Df}}.
		\end{align*}
        Finally, if we had
        \[
			\liminf_{\lambda\to \infty}F_{\gamma,\lambda}(u_{\lambda},\Om)
			< \frac{C_n}{\gamma}|D^a u|(\Om)
			+ \frac{C_n}{\gamma+1}|D^s u|(\Om),
		\]
        then for some $\eps>0$ we could find a subsequence $\lambda_j'\to\infty$ such that
        \[
        F_{\gamma,\lambda'_j}(u_{\lambda'_j},\Om)
		\le \frac{C_n}{\gamma}|D^a u|(\Om)
		+ \frac{C_n}{\gamma+1}|D^s u|(\Om)-\eps
        \]
        for all $j\in\N$. But then by the above argument, we could find a further subsequence
        $\lambda_j\to \infty$ such that
        \[
        \liminf_{j\to \infty}F_{\gamma,\lambda_j}(u_{\lambda_j},\Om)
			\ge \frac{C_n}{\gamma}|D^a u|(\Om)
			+ \frac{C_n}{\gamma+1}|D^s u|(\Om),
        \]
        which is a contradiction.
	\end{proof}

	\section{Upper bound}\label{sec:upper bound}
	
	%The $\Gamma$-limit of the functionals $F_{\gamma,\lambda}$ (with respect to
	%$L^1_{\loc}$-convergence) as $\lambda\to \infty$ is not known.
	In this section we consider the upper bound for the $\Gamma_{\mathrm{AS}}$-limit.

    Just as in the case of the lower bound, we rely on certain estimates known for a fixed function $u$.
    In particular, by Picenni \cite[Theorem 1.2]{Pic} we know that for $u\in \dot{\SBV}(\R^n)$, there holds
    \begin{equation}\label{eq:SBV limit2}
    \lim_{\lambda\to\infty}F_{\gamma,\lambda}(u,\R^n)
	=  \frac{C_n}{\gamma}|D^a u|(\R^n)
	+ \frac{C_n}{\gamma+1}|D^s u|(\R^n).
    \end{equation}
    
	We have the following generalization of this fact.
    
	\begin{theorem}\label{thm:SBV}
	Let $\Om\subset \R^n$ be either $\R^n$ or a bounded domain with Lipschitz
	boundary. Suppose that $u\in \dot{\SBV}(\Om)$.
	Then
	\[
	\lim_{\lambda\to\infty}F_{\gamma,\lambda}(u,\Om)
	=  \frac{C_n}{\gamma}|D^a u|(\Om)
	+ \frac{C_n}{\gamma+1}|D^s u|(\Om).
	\]
	\end{theorem}
	\begin{proof}
    Due to \eqref{eq:sharp lower bound} and \eqref{eq:SBV limit2},
    we only need to prove that
     \begin{equation}\label{eq:SBV upper limit}
    \limsup_{\lambda\to\infty}F_{\gamma,\lambda}(u,\Om)
	\le  \frac{C_n}{\gamma}|D^a u|(\Om)
	+ \frac{C_n}{\gamma+1}|D^s u|(\Om)
    \end{equation}
	when $\Omega\subset\R^n$ is a bounded domain with Lipschitz boundary.
    Consider such an $\Om$, and we know it supports a Poincar\'e inequality,
    see e.g. \cite[Proposition 3.21, Theorem 3.44]{AFP}.
    Then it follows (see e.g. \cite[Eq. (3.3)]{BB}) that
    $u\in L^1(\Om)$ and thus $u\in \SBV(\Om)$.

    First consider an open interval $I\subset \R$ and $v\in \dot{\SBV}(I)$.
    We can extend $v$ to $w\in \dot{\SBV}(\R)$ with $|Dw|(\R)=|Dv|(I)$. Then
    by \eqref{eq:SBV limit2}, we have
     \begin{equation}\label{eq:I and R}
    \begin{split}
    \limsup_{\lambda\to\infty}F_{\gamma,\lambda}(v,I)
    \le \limsup_{\lambda\to\infty}F_{\gamma,\lambda}(w,\R)
	&= \frac{2}{\gamma}|D^a w|(\R) + \frac{2}{\gamma+1}|D^s w|(\R)\\
    &=\frac{2}{\gamma}|D^a v|(I) + \frac{2}{\gamma+1}|D^s v|(I).
    \end{split}
    \end{equation}
     
	Returning to the $n$-dimensional situation, we can extend $u$ to a function
    (not relabelled) $u\in \SBV(\R^n)$ with $|Du|(\partial\Om)=0$;
    this is shown e.g. in \cite[Proposition 3.21]{AFP}, though without stating the property that
    the extension has no Cantor part, but this follows since the extension is constructed using a simple
    reflection argument.
    Fix $0<\eps<1$.
    Choose an open set $\Om'$
    with $\Om\Subset \Om'$ and
    \[
    |Du|(\Om'\setminus \Om)<\eps.    
    \]
    Given any $\sigma\in \mathbb S^{n-1}$ 
    and $z\in\langle \sigma\rangle^\perp$, we have that
    $(\overline{\Om})_{\sigma,z}\subset (\Om')_{\sigma,z}$.
    Note that $(\Om')_{\sigma,z}$ is an open subset of $\R$,
    and thus is an at most countable collection of pairwise disjoint intervals.
    Denote these by $I_{\sigma,z,j}$, $j\in\N$.
    Since $(\overline{\Om})_{\sigma,z}$ is compact, there is a finite subcollection
    $\{I_{\sigma,z,j}\}_{j=1}^{M_{\sigma,z}}$ with some finite $M_{\sigma,z} \in \N$ which still contains $(\overline{\Om})_{\sigma,z}$.
    We find open
    intervals $I_{\sigma,z,j}'\Subset I_{\sigma,z,j}$ whose union still contains
    $(\overline{\Om})_{\sigma,z}$.
    Note that these intervals are at a strictly positive distance $\delta_{\sigma,z}>0$ from each other.
    Fix $N\in\N$ and let $u_N:=\max\{\min\{u,N\},-N\}$.
    Given $x \in I'_{\sigma,z,i}$, $y \in I'_{\sigma,z,j}$ with $1\leq i < j \leq M_{\sigma,z}$, we have
    \[ \frac{\abss{u_N(z+\sigma x)-u_N(z+\sigma y)}}{\abss{x-y}^{1+\gamma}} \leq \frac{2N}{\delta_{\sigma,z}^{1+\gamma}} <\lambda \]
    for $\lambda>0$ that is sufficiently large.
    Then there holds
    \begin{equation}\label{eq:sum for slices}
    \begin{split}
    &\ \nu_{\gamma}(E_{\gamma,\lambda}(u_N,\Omega_{\sigma,z})) \\
    =&\ \int_{\Om_{\sigma,z}}\int_{\Om_{\sigma,z}}\mathbbm{1}_{\{(x,y)\colon
				|u_N(z+\sigma x)-u_N(z+\sigma y)|/|x-y|^{1+\gamma}>\lambda\}}(x,y)
			\cdot |x-y|^{\gamma-1}\id x\id y\\
    \le &\ \sum_{j=1}^{M_{\sigma,z}}\int_{I_{\sigma,z,j}'}\int_{I_{\sigma,z,j}'}
            \mathbbm{1}_{\{(x,y)\colon
				|u_N(z+\sigma x)-u_N(z+\sigma y)|/|x-y|^{1+\gamma}>\lambda\}}(x,y)
			\cdot |x-y|^{\gamma-1}\id x\id y\\
    =&\ \sum_{j=1}^{M_{\sigma,z}} \nu_{\gamma}(E_{\gamma,\lambda}(u_N,I'_{\sigma,z,j})).
    \end{split}
    \end{equation}
    Note that
    \begin{align*}
%        &\lambda\int_{\Om_{\sigma,z}}\int_{\Om_{\sigma,z}}\mathbbm{1}_{\{(x,y)\colon
%				|u_N(z+\sigma x)-u_N(z+\sigma y)|/|x-y|^{1+\gamma}>\lambda\}}(x,y)
%			\cdot |x-y|^{\gamma-1}\id x\id y\\
%        & \quad \le \lambda\int_{\R}\int_{\R}\mathbbm{1}_{\{(x,y)\colon
%				|u_N(z+\sigma x)-u_N(z+\sigma y)|/|x-y|^{1+\gamma}>\lambda\}}(x,y)
%			\cdot |x-y|^{\gamma-1}\id x\id y\\
%        & \quad \le C'|D u_{\sigma,z}|(\R)   
        \lambda \nu_{\lambda}(E_{\gamma,\lambda}(u_N,\Omega_{\sigma,z})) = F_{\gamma,\lambda}(u_N,\Omega_{\sigma,z}) \le F_{\gamma,\lambda}(u_N,\R) \le C'|D u_{\sigma,z}|(\R)
    \end{align*}
    for all $\lambda >0$ by \eqref{eq:sup bound}.
    This is integrable over $z\in \langle \sigma\rangle^\perp$
    by \eqref{eq: one dim Df int}
    %\eqref{eq:change of var for Df}
    (with $D^a$ replaced by $D$),
    which justifies the use of Fatou's lemma below.
    Moreover, note that for a.e. $\sigma\in \mathbb S^{n-1}$ and a.e. $z\in \langle \sigma\rangle^\perp$,
    by \eqref{eq:change of var for Df} we have
    \begin{equation}\label{eq:no Cantor}
    |D^c u_{\sigma,z}|(\R)=0.
    \end{equation}
    By \eqref{eq:Om multiple integral} we have
		\begin{align*}
%			&\limsup_{\lambda\to\infty}\lambda \nu_\gamma\left(\left\{(x',y')\in \Om\times \Om\colon
%			\frac{|u_{N}(x')-u_{N}(y')|}{|x'-y'|^{1+\gamma}}>\lambda\right\}\right)\\
            & \limsup_{\lambda\to\infty} F_{\gamma,\lambda}(u_N,\Omega)\\
			& =\limsup_{\lambda\to\infty}\lambda\int_{\Om}\int_{\Om} \mathbbm{1}_{\{(x',y')\colon
				|u_{\lambda}(x')-u_{N}(y')|/|x'-y'|^{1+\gamma}>\lambda\}}(x',y')|x'-y'|^{\gamma-n}\id x'\id y'\\
			&= \limsup_{\lambda\to\infty}\frac {\lambda}{2}\int_{\mathbb S^{n-1}}\int_{\langle \sigma\rangle^\perp} \int_{\Om_{\sigma,z}}\int_{\Om_{\sigma,z}}\\
            & \hspace{1.7cm}\mathbbm{1}_{\{(x,y)\colon
				|u_{N}(z+\sigma x)-u_{N}(z+\sigma y)|/|x-y|^{1+\gamma}>\lambda\}}(x,y) \cdot |x-y|^{\gamma-1}\id x\id y\id z\id \sigma \\
            &=\limsup_{\lambda\to\infty}\frac {1}{2}\int_{\mathbb S^{n-1}}\int_{\langle \sigma\rangle^\perp}  F_{\gamma,\lambda}(u_N, \Omega_{\sigma,z})\id z\id \sigma.
        \end{align*} With Fatou's lemma and \eqref{eq:sum for slices}, the
        %integral in the
        last line can be further controlled by the following
        \begin{align*}
%			&\frac {1}{2}\int_{\mathbb S^{n-1}}\int_{\langle \sigma\rangle^\perp}\limsup_{\lambda\to\infty}\lambda\int_{\Om_{\sigma,z}}\int_{\Om_{\sigma,z}}\\
%			&\quad\ \mathbbm{1}_{\{(x,y)\colon
%				|u_{N}(z+\sigma x)-u_{N}(z+\sigma y)|/|x-y|^{1+\gamma}>\lambda\}}(x,y)
%			\cdot |x-y|^{\gamma-1}\id x\id y\id z\id \sigma\\
%           &\le \frac {1}{2}\int_{\mathbb S^{n-1}}\int_{\langle \sigma\rangle^\perp}\limsup_{\lambda\to\infty}\sum_{j=1}^{M_{\sigma,z}}
%          \lambda\int_{I_{\sigma,z,j}'}\int_{I_{\sigma,z,j}'}\\
%			&\quad\ \mathbbm{1}_{\{(x,y)\colon
%				|u_{N}(z+\sigma x)-u_{N}(z+\sigma y)|/|x-y|^{1+\gamma}>\lambda\}}(x,y)
%			\cdot |x-y|^{\gamma-1}\id x\id y\id z\id \sigma\\
            &\frac {1}{2}\int_{\mathbb S^{n-1}}\int_{\langle \sigma\rangle^\perp}\limsup_{\lambda\to\infty} F_{\gamma,\lambda}(u_N,\Omega_{\sigma,z}) \id z \id \sigma\\
            &\le \frac {1}{2}\int_{\mathbb S^{n-1}}\int_{\langle \sigma\rangle^\perp}\limsup_{\lambda\to\infty}\sum_{j=1}^{M_{\sigma,z}} F_{\gamma,\lambda}(u_N,I'_{\sigma,z,j})\id z\id \sigma\\
			& \le \int_{\mathbb S^{n-1}}\int_{\langle \sigma\rangle^\perp}
			\sum_{j=1}^{M_{\sigma,z}}\left[\frac{1}{\gamma}|D^a u_{\sigma,z}|(I_{\sigma,z,j}')
			+\frac{1}{\gamma+1}
			|D^j u_{\sigma,z}|(I_{\sigma,z,j}')\right]\id z\id \sigma
        \end{align*}
        by \eqref{eq:I and R} and \eqref{eq:no Cantor}.
        Combining the two estimates above, we have
        \begin{align*}
%            &\limsup_{\lambda\to\infty}\lambda \nu_\gamma\left(\left\{(x',y')\in \Om\times \Om\colon
%			\frac{|u_{N}(x')-u_{N}(y')|}{|x'-y'|^{1+\gamma}}>\lambda\right\}\right)\\
            & \limsup_{\lambda\to\infty} F_{\gamma,\lambda}(u_N,\Omega)\\
            & \le \int_{\mathbb S^{n-1}}\int_{\langle \sigma\rangle^\perp}
			\left[\frac{1}{\gamma}|D^a u_{\sigma,z}|((\Om')_{\sigma,z})
			+\frac{1}{\gamma+1}
			|D^j u_{\sigma,z}|((\Om')_{\sigma,z})\right]\id z\id \sigma\\
			& = \frac{C_n}{\gamma}|D^a u|(\Om')+\frac{C_n}{\gamma+1}|D^j u|(\Om')
            \quad\textrm{by }\eqref{eq:change of var for Df}\\
            &\le \frac{C_n}{\gamma}|D^a u|(\Om)+\frac{C_n}{\gamma+1}|D^j u|(\Om)
            +\frac{C_n}{\gamma}\eps.
		\end{align*}
    On the other hand, we have
	\begin{align*}
		&\left\{(x,y)\in \Om\times \Om\colon
		\frac{|u(x)-u(y)|}{|x-y|^{1+\gamma}}>\lambda\right\}\\
		&\quad\subset  \left\{(x,y)\in \Om\times \Om\colon
		\frac{|u_N(x)-u_N(y)|}{|x-y|^{1+\gamma}}>(1-\eps)\lambda\right\}\\
		&\qquad \cup  \left\{(x,y)\in \R^n\times \R^n\colon
		\frac{|(u-u_N)(x)-(u-u_N)(y)|}{|x-y|^{1+\gamma}}
		>\eps\lambda\right\}.
	\end{align*}
	Thus
	\begin{align*}
		&\limsup_{\lambda \to \infty}F_{\gamma,\lambda}\left(u,\Om\right)\\
		&\qquad \le (1-\eps)^{-1}\limsup_{\lambda \to \infty}F_{\gamma,(1-\eps)\lambda}\left(u_N,\Om\right)
		+\eps^{-1}\limsup_{\lambda \to \infty}F_{\gamma,\eps\lambda}\left(u-u_N,\R^n\right)\\
		&\qquad\le (1-\eps)^{-1}\left(\frac{C_n}{\gamma}|D^a u|(\Om)+\frac{C_n}{\gamma+1}|D^j u|(\Om)
            +\frac{C_n}{\gamma}\eps\right)+\eps^{-1}\frac{C_n}{\gamma}|D(u-u_N)|(\R^n)
	\end{align*}
    by \eqref{eq:SBV limit2}.
    Letting $N\to\infty$ and then $\eps\to 0$, we get the result.
	\end{proof}

    The following lemma follows from \cite{KrRind2}.
	
	\begin{lemma}\label{lem:mollifying}
	Let $\Om\subset \R^n$ be open and let $u\in \dot{\BV}(\Om)$. Then there exists a sequence $\{v_i\}_{i=1}^{\infty}$ in $\dot{W}^{1,1}(\Om) \cap C^{\infty}(\Om)$ such that 
        \begin{equation*}
            v_i \to u \quad \mbox{in }L^1_{\loc}(\Om),\quad  
            |Dv_i|(\Om)\to |Du|(\Om),
            \quad  E(Dv_i)(\Om) \to E(Du)(\Om),
        \end{equation*} and the functions 
        \[ 
        \tilde{v}_i(x) := \left\{
        \begin{aligned}
        &v_i(x)-u(x),& &x\in \Omega,\\
        &0,& &x\in \R^n \setminus\Omega
        \end{aligned}\right. 
        \]
        lie in the space $\dot{\BV}(\R^n)$ with $\abss{D \tilde{v}_i}(\R^n \setminus \Omega)=0$,
        and $\tilde{v}_i \to 0$ in $L^1_{\loc}(\R^n)$.
       % Moreover, if $u \in \BV(\Om)$, the sequence $\{v_i\}$ can be taken in $W^{1,1}(\Om)$.
	\end{lemma}
    \begin{proof}
    This is given in \cite[Lemma 1]{KrRind2} for bounded open sets; however, the proof therein
    applies to any open set in $\R^n$.
    The facts that $|Dv_i|(\Om)\to |Du|(\Om)$ and
    $\tilde{v}_i \to 0$ in $L^1_{\loc}(\R^n)$ are not
    given there, but the first one follows from the area-strict convergence and Theorem 5 in the same reference, and the second
    one can be obtained as follows.
    We already know that $\tilde{v}_i \to 0$ in $L^1_{\loc}(\Om)$;
    we can assume that moreover $\tilde{v}_i(x)\to 0$ for a.e. $x\in \R^n$.
    On the other hand, given $x\in \R^n\setminus \Om$ and $r>0$, we note that for large $i\in\N$,
    $|\tilde{v}_i|\le 1$ in some set $A\subset B(x,r)$ with
    $\Le^n(A)\ge \Le^n(B(x,r))/2$,
    and then we can apply a Poincar\'e inequality
    (see e.g. \cite[Lemma 2.2]{KKSL13}) to $(\abss{\tilde{v}_i}-1)_+$ to obtain
       \begin{align*}
			\int_{B(x,r)}|\tilde{v}_i|\id y
			&\le Cr |D\tilde{v}_i|(B(x,r)) +\Le^n(B(x,r))\\
			&= Cr |D\tilde{v}_i|(B(x,r)\cap \Om)+\Le^n(B(x,r))\\
			&\le  Cr (|Du|(\Om)+|Dv_i|(\Om))+\Le^n(B(x,r))
		\end{align*}
        for a constant $C$ independent of $i$. Since $|Dv_i|(\Om)$ is a bounded sequence,
        it follows that $\tilde{v}_i$ is bounded in $\BV(B(x,r))$,
        and thus by BV compactness, we find a subsequence of $\tilde{v}_i$ (not relabeled)
        converging in $L^1(B(x,r))$
        to a limit, which is necessarily $0$.
        We can do this for each $r=j$, $j\in\N$, and
        then by a diagonalization argument, we obtain the convergence
        $\tilde{v}_i \to 0$ in $L^1_{\loc}(\R^n)$ for a further subsequence  (not relabeled).
    \end{proof}
    
	We have also the following standard approximation lemma.
	
	\begin{lemma}\label{lem:appr with jump functions}
	Let $\Omega\subset\R^n$ be open  and bounded, and let $u\in \BV(\Om)$.
	Then there exists a sequence of functions
	$w_i\in \BV(\Om)$ with
    %$h_i- g\to 0$ in $L^{\infty}_{\loc}(\Om)$, in particular
    $w_i\to u$ in $L^{1}(\Om)$ and
	\[
	|D^j w_i|(\Om)=|Dw_i|(\Om)\to |Du|(\Om)\quad \textrm{as }i\to\infty.
	\]
	\end{lemma}
    \begin{proof}
        %\par We first assume that $\Omega$ is a bounded open set.
        The coarea formula (see, e.g., Theorem 5.9 in \cite{EG}) gives
        \[
        \abss{Du}(\Omega) = \int_{-\infty}^{\infty} \abss{D\mathbbm{1}_{E_t}}(\Omega)\id t,
        \]
        where $E_t:= \{x\in \Omega\col u(x)>t\}$.
        Fix $\varepsilon>0$. Then it is possible to take $M\gg 1$ such that 
        \begin{gather}
            \mbox{both } E_{M} \mbox{ and }\{x\in \Omega\col u(x)<-M\} \mbox{ have finite perimeters in }\Omega \label{eq:jump approx trunc1}\\
            \quad 
            \mbox{and}\quad
            \int_{\{x\in \Omega\col \abss{u(x)}\geq M-1\}} \abss{u} \id x <\eps.%\quad 
            %\mbox{and}\quad \brac{\int_{-\infty}^{-M_{\eps}}+\int_{M_{\eps}}^{\infty}} 
            %\abss{D\mathbbm{1}_{E_t}} (\Omega)\id t <\eps.
            \label{eq:jump approx trunc2}
        \end{gather}
        \par Let $k$ be a positive integer
        to be determined later, and define
        \begin{gather*}
        I_0:= \big(-M,-M+\frac{2M}{k}\big), \\
        I_j:= \big[-M+j\frac{2M}{k},-M+(j+1)\frac{2M}{k}\big),\ j=1,\dots,k-1.
        \end{gather*} 
        Select numbers $t_j \in I_j$ such that $E_{t_j}$ has finite perimeter in $\Omega$ with
        \begin{equation}\label{eq:jump approx Rsum}
            \sum_{j=0}^{k-1}\frac{2M}{k}\abss{D\mathbbm{1}_{E_{t_j}}}(\Omega) 
            \le \int_{-M}^{M} \abss{D\mathbbm{1}_{E_{t}}}(\Om)\id t.
        \end{equation}
        %which is possible when $k$ is large enough.
        \par Then define 
        \[
        u_{\eps}:= -M\mathbbm{1}_{\Omega}+\sum_{j=0}^{k-1} \frac{2M}{k}\mathbbm{1}_{E_{t_j}}.
        \]
        By the boundedness of $\Omega$ and the choice of $t_j$, we know that $u_{\eps} \in \BV(\Omega)$ with 
        \[
        Du_{\eps} = D^j u_{\eps} = \sum_{j=0}^{k-1} \frac{2M}{k}D \mathbbm{1}_{E_{t_j}}\mres \Omega.
        \]
        \par The definition of $u_{\eps}$ indicates that 
        \begin{gather}
        \abss{u_{\eps}-u} \leq \frac{2M}{k} \quad \mbox{in }\{x\in \Omega\col t_0<u(x)\leq t_{k-1}\},\\
        \mbox{and}\quad \abss{u_{\eps}} \leq \abss{u}+\frac{2M}{k} \quad \mbox{in } E_{t_{k-1}}\cup (\Omega\setminus E_{t_0}).
        \end{gather}
        Thus, the $L^1$-distance between $u_{\eps}$ and $u$ is controlled as follows:
        \begin{align*}
        \int_{\Omega} \abss{u_{\eps}-u}\id x 
        &\leq \int_{E_{t_{k-1}}\cup (\Omega\setminus E_{t_0})} \brac{2\abss{u}+\frac{2M}{k}} \id x + \int_{\Omega}\frac{2M}{k} \id x\\
        &\leq \int_{\{x\in \Omega\col \abss{u(x)}\geq M-1\}} 2\abss{u}\id x + \Le^n (\Omega) \frac{4M}{k} \leq 4\varepsilon,
        \end{align*}
        if $k$ is chosen such that $2M/k <\min \{1, \eps/\Le^n(\Omega)\}$. On the other hand, by the coarea formula we have that 
        \[
        \abss{Du_{\eps}}(\Omega) = \sum_{j=0}^{k-1} \frac{2M}{k} \abss{D\mathbbm{1}_{E_{t_j}}}(\Omega),
        \]
        and thus \eqref{eq:jump approx Rsum} and the coarea formula imply that
        $\abss{Du_{\eps}}(\Omega)\le \abss{Du}(\Omega)$.
        Recalling also the lower semicontinuity of the total variation with respect to $L^1$ convergence,
        we have that the sequence
        $\{w_i\}$ defined by $w_i := u_{1/i}$ satisfies the desired properties.
        %\par In the general case, take a smooth approximating sequence $\{g_i\}$ as in
        %Lemma \ref{lem:mollifying}. For each $g_i$, take an exhaustion $\{\Omega_{i,\ell}\}_{\ell\in\N}$
        %of $\Omega$ in such a way that
        %$\int_{\partial \Omega_{i,\ell}} \abss{g_i}\id \mathscr{H}^{n-1} < 1/(2^{\ell}i)$.
        %Then on each $\Omega_{i,\ell+1}\setminus\Omega_{i,\ell}$, approximate $g_i$ as above
        %with a truncation number 
        %\[ 
        %M_{i,\ell}> \max\{\sup_{\partial \Omega_{i,\ell+1}}\abss{g_i}, \sup_{\partial
        %\Omega_{i,\ell+1}}\abss{g_i}\}
        %\]
        %and a large integer $k_{i,\ell}$ such that the variation on $\partial \Omega_{i,\ell}\cup\partial \Omega_{i,\ell+1}$ is small.
    \end{proof}
	
	Now we prove the main result of this section.
	
	\begin{theorem}\label{thm:Gamma area strict upper bound}
		Let $\Om\subset \R^n$ be either $\R^n$ or a bounded domain with Lipschitz boundary. 
		Suppose $u\in \dot{\BV}(\Om)$. Then there exists a family
		$\{u_\lambda\}_{\lambda>0}$ with $u_\lambda\to u$ area-strictly as $\lambda\to\infty$, and such that
		\begin{align*}
			\limsup_{\lambda\to \infty}F_{\gamma,\lambda}(u_{\lambda},\Om)
			\le \frac{C_n}{\gamma}|D^a u|(\Om)
			+ \frac{C_n}{\gamma+1}|D^s u|(\Om).
		\end{align*}
	\end{theorem}
	\begin{proof}
	Fix $\eps>0$ and take a $\Le^n$-negligible compact set $K\subset \Om$ such that
	\begin{equation}\label{eq:choice of K}
	|D^s u|(\Om\setminus K)<\eps/2
	\quad\textrm{and}\quad
	|D^a u|(K)=0.
	\end{equation}
	Then choose a bounded open set $W\Subset \Om$ containing $K$ with 
	\[
    |Du|(W\setminus K)<\eps/2, \quad \Le^n (\partial W)=0 = \abss{Du}(\partial W).
    \]
	It follows that
	\begin{equation}\label{eq:choice of W}
	|Du|(W)
	= |D^su|(K)+|Du|(W\setminus K)
	< |D^s u|(\Om)+\eps/2.
	\end{equation}
	
	Apply Lemma \ref{lem:mollifying} to $u \in \dot{\BV}(\Om\setminus K)$ to obtain
    functions $v_i\in \dot{W}^{1,1}(\Om\setminus K)\cap C^{\infty}(\Om\setminus K)$
    converging to $u$ strictly and area-strictly in $\Om\setminus K$.
    We can extend each $v_i$ by $u$ to $K$, and by the lemma the extended functions (not relabelled)
    are in $\dot{\BV}(\Omega)$,
    and satisfy $v_i\to u$ in $L_{\loc}^1(\Om)$ and $|D(v_i-u)|(K)=0$.
    Then it is clear that the functions $v_i$ converge to $u$ strictly and
    area-strictly in $\Om$.
    %and then by Lemma \ref{lem:strict and area} also in $W\setminus K$ as well as in $W$ (we can interpret $g$ to be defined in
    %the whole of $\Om$),  since $|D(f-g)|(K)=0=|Df|(\partial W)$ and
    %$\Le^n (K)=0=\Le^n (\partial W)$.
   \par Take an open set $\Omega_{\eps}$ with $W \Subset \Om_{\eps} \Subset \Om$. 
	Choosing $v=v_i$ for a sufficiently large $i$, we have the following:
    %also with
    %$g\in W^{1,1}(W\setminus K)$.
    \begin{enumerate}[label=(\arabic*)]
    \item $v \in \dot{\BV}(\Omega)$ and 
	$\Vert v-u\Vert_{L^1(\Om_\eps)}<\eps$; 
	\item by \eqref{eq:choice of K},
    \begin{equation}\label{eq:choice of g initial}
       \begin{split}
		 \int_{\Om\setminus K}\sqrt{1+|\nabla v|^2}\,dx %+|D^s v|(\Om\setminus K) 
		&\le	\int_{\Om\setminus K}\sqrt{1+|\nabla u|^2}\,dx+|D^s u|(\Om\setminus K)
		+\eps/2 \\
		&<	\int_{\Om\setminus K}\sqrt{1+|\nabla u|^2}\,dx+\eps;
        \end{split}
    \end{equation}
    \item  by the strict convergence and \eqref{eq:choice of K},
    \begin{equation}\label{eq:choice of i}
	\int_{\Om}|\nabla v|\,dx= |Dv|(\Om\setminus K)
    \le |Du|(\Om\setminus K)+\eps/2
    <\int_{\Om}|\nabla u|\,dx+\eps;
	\end{equation}
	\item by \eqref{eq:strict conv consequence two} and the fact that
    $\Le^n (\partial W)=0 = \abss{Du}(\partial W)$, we get
	\begin{equation}\label{eq:choice of g}
	|D v|(W)
    %\le |D v|(\overline{W})
    \le |Du|(W)+\eps/2.
	\end{equation}
    \end{enumerate}
    %Now take $v=v_i$ with $v_i$ satisfying the above.
	Using Lemma \ref{lem:appr with jump functions},
	we find a sequence $w_i\in \BV(W)$ with $w_i\to v$ in $L^1(W)$ and
	\begin{equation}\label{eq:choice of hi}
	|D^j w_i|(W)=|Dw_i|(W)\to |Dv|(W).
	\end{equation}
    Take a cutoff function $\eta\in C_c^{\infty}(W)$ with $0\le \eta\le 1$ on $W$,
    $\eta=1$ in $K$.
	Then define
	\[
	u_i':=\eta w_i+(1-\eta)v.
	\]
	By  the Leibniz rule for BV functions,
    see e.g. \cite[Example 3.97]{AFP},
    we have
    \[
    D u_i' = \eta D^j w_i +(1-\eta)\nabla v +\nabla \eta (w_i-v).
    \] Thus, it follows that
	\begin{align}
	&|Du_i'|
	\le |Dw_i|\mres W+|Dv|\mres (\Om\setminus K)
	+ |w_i-v||\nabla \eta|, \notag\\
	&|\nabla u_i'|
	\le |\nabla v|\mathbbm{1}_{\Omega\setminus K}+ |w_i-v||\nabla \eta|,\notag\\
    &|D^c u_i'|(\Om)=0, \quad \textrm{and}\quad |D^j u_i'|
	\le |D^j w_i|\mres W. \label{eq:nabla f i} %+|D^jv|\mres (\Om\setminus K).
	\end{align}
	Then we have
    \begin{align*}
		\limsup_{i\to\infty}|D^j(u_i')|(\Om)
		&\le \limsup_{i\to\infty}|D^jw_i|(W)\\%+|D^j v|(\Om\setminus K)\\
		%&=|Dv|(W)+0
		%\quad\textrm{by }\eqref{eq:choice of hi}\textrm{ and since }v\in C^{\infty}(\Om\setminus K)\\
		&\le |Du|(W)+\eps/2
		\quad\textrm{by }\eqref{eq:choice of hi},\,\eqref{eq:choice of g}\\
		&< |D^s u|(\Om)+\eps
		\quad\textrm{by }\eqref{eq:choice of W},
	\end{align*}
	as well as
	\begin{align*}
		\limsup_{i\to\infty}\int_{\Om}\sqrt{1+|\nabla u_i'|^2}\,dx
		&=\limsup_{i\to\infty}\int_{\Om\setminus K}\sqrt{1+|\nabla u_i'|^2}\,dx\\
		&\le \int_{\Om\setminus K}\sqrt{1+|\nabla v|^2}\,dx\quad 
		\textrm{by }\eqref{eq:nabla f i}\\
		&< \int_{\Om\setminus K}\sqrt{1+|\nabla u|^2}\,dx+\eps
		\quad\textrm{by }\eqref{eq:choice of g initial},
	\end{align*}
    and similarly, by \eqref{eq:choice of i},
    \[
    \limsup_{i\to\infty}\int_{\Om}|\nabla u_i'|\,dx
    <\int_{\Om}|\nabla u|\,dx+\eps.
    \]
	Taking $u'$ to be $u'_i$ for sufficiently large $i\in\N$,
	 we have $u'\in \dot{\SBV}(\Om)$ with $\Vert u'-u\Vert_{L^1(\Om_{\eps})}<\eps$,
	%$|Df'|(\Om)<|Df|(\Om)+\eps$,
    \begin{equation}\label{eq:jump upper bound}
	|D^s u'|(\Om)\le |D^s u|(\Om)+\eps,
    \end{equation}
	and
	\begin{equation}\label{eq:abs cont upper bound }
	\int_{\Om}\sqrt{1+|\nabla u'|^2}\,dx 
	\le \int_{\Om}\sqrt{1+|\nabla u|^2}\,dx+\eps
    \quad\textrm{and}\quad
    \int_{\Om}|\nabla u'|\,dx\le \int_{\Om}|\nabla u|\,dx+\eps.
	\end{equation}
	With the choice $\eps=1/k$, label $u_k:=u'$. The open sets
    $\{\Om_{1/k}\}$ can be taken as an exhaustion of $\Om$, that is, 
    $\Om_{1/k} \Subset \Om_{1/(k+1)}$ and $\bigcup_{k=1}^{\infty} \Omega_{1/k} = \Omega$.
	The sequence $\{u_k\}$ chosen as above satisfies the following:
    \begin{align*}
    u_k \in \dot{\SBV}(\Om),\quad
    u_k \to u \quad \mbox{in }L^1_{\loc}(\Om)\quad
    %\quad Df_k \weaksto Df \quad \mbox{in }\M(\Omega),\\
    \mbox{and}\quad  \limsup_{j\to \infty} E(Du_k)(\Om) \leq E(Du)(\Om).
    \end{align*} The functional $\int_{\Om}E(\cdot)$ is lower semicontinuous with respect to weak$^{\ast}$ convergence, which further implies that $u_{k}\to u$ area-strictly in $\Om$.
    By \eqref{eq:jump upper bound} and \eqref{eq:abs cont upper bound }, we get
    \begin{equation}\label{eq:two convergences}
        |D^j u_k|(\Om)=|D^s u_k|(\Om)\le |D^s u|(\Om)+1/k
        \quad\textrm{and}\quad
        |D^a u_k|(\Om)\le |D^a u|(\Om)+1/k.
    \end{equation}
	
	For every $k\in\N$, by Theorem \ref{thm:SBV}
	we find $\lambda_k>k$ such that 
	\[
	F_{\gamma,\lambda}(u_k,\Om)
	\le  \frac{C_n}{\gamma}|D^a u_k|(\Om)
	+ \frac{C_n}{\gamma+1}|D^s u_k|(\Om)+\frac 1k
	\]
	for all $\lambda\ge \lambda_k$.
	We can also assume that $\lambda_{k+1}>\lambda_k$.
	Define $u_{\lambda}:=u_k$ for all $\lambda\in [\lambda_k,\lambda_{k+1})$.
	Note that $k\to \infty$ as $\lambda\to \infty$, and so we get 
	$u_{\lambda}\to u$ area-strictly as $\lambda \to \infty$.
	For all $\lambda\in [\lambda_k,\lambda_{k+1})$, we have
	\begin{align*}
		F_{\gamma,\lambda}(u_\lambda,\Om)
		&= F_{\gamma,\lambda}(u_k,\Om) \\
		&\le \frac{C_n}{\gamma}|D^a u_k|(\Om)
		+ \frac{C_n}{\gamma+1}|D^s u_k|(\Om)+\frac 1k\\
		&\le \frac{C_n}{\gamma}(|D^a u|(\Om)+1/k)
		+ \frac{C_n}{\gamma+1}(|D^s u|(\Om)+1/k)+\frac 1k
	\end{align*}
    by \eqref{eq:two convergences}.
	Thus
	\[
	\limsup_{\lambda\to \infty}F_{\gamma,\lambda}(u_\lambda,\Om)
	\le \frac{C_n}{\gamma}|D^a u|(\Om)
	+ \frac{C_n}{\gamma+1}|D^s u|(\Om).
	\]
	\end{proof}

\begin{proof}[Proof of Theorem \ref{thm:Gamma area strict limit}]
Combine Theorems \ref{thm:Gamma area strict lower bound} and \ref{thm:Gamma area strict upper bound}.
\end{proof}

The recovery sequence can be chosen to consist of smooth functions.

\begin{proposition}
		Let $\Om\subset \R^n$ be either $\R^n$ or a bounded domain with Lipschitz boundary. 
		Suppose $u\in \dot{\BV}(\Om)$. Then there exists a family
		$\{u_\lambda\}_{\lambda>0}\subset C^{\infty}(\Om)$
        with $u_\lambda\to u$ area-strictly as $\lambda\to\infty$, and such that
		\begin{align*}
			\limsup_{\lambda\to \infty}F_{\gamma,\lambda}(u_{\lambda},\Om)
			\le \frac{C_n}{\gamma}|D^a u|(\Om)
			+ \frac{C_n}{\gamma+1}|D^s u|(\Om).
		\end{align*}
	\end{proposition}
	\begin{proof}
    Assume first that $u$ is bounded, so that $|u|< M$ for some $0<M<\infty$.
    Take the sequence $\{u_\lambda\}_{\lambda>0}$ from
    Theorem \ref{thm:Gamma area strict upper bound}.
    In the case $\Om=\R^n$, note that by the construction, every function
    $u_\lambda$ is smooth outside some closed ball
    $\overline{B}(0,R_{\lambda})$.
    Then for every $\lambda>0$, from Lemma \ref{lem:mollifying} we find a sequence
    $\{v_{\lambda,i}\}_{i=1}^{\infty}$ such that
	$v_{\lambda,i}\in \dot{W}^{1,1}(B(0,R_{\lambda}+1)\cap C^{\infty}(B(0,R_{\lambda}+1))$,
    and $v_{\lambda,i}\to u_{\lambda}$ area-strictly in $B(0,R_{\lambda}+1)$.
    We can also assume that $v_{\lambda,i}\to u_{\lambda}$ a.e. in $B(0,R_{\lambda}+1)$.
    Take a cutoff function $\eta_{\lambda}\in C_c^{\infty}(B(0,R_{\lambda}+1))$ with
    $0\le \eta_{\lambda}\le 1$ and $\eta_{\lambda}=1$ in $B(0,R_{\lambda})$.
    Then define
    \[
    u_{\lambda,i}:=\eta_{\lambda}v_{\lambda,i}+(1-\eta_{\lambda})u_{\lambda}.
    \]
    By the Leibniz rule for BV functions, it is easy to check that
	$u_{\lambda,i}\in \dot{W}^{1,1}(\Om)\cap C^{\infty}(\Om)$,
    $u_{\lambda,i}- u_{\lambda}\to 0$ in $L^1(\Om)$,
    $u_{\lambda,i}\to u_{\lambda}$ area-strictly in $\Om$,
    and $u_{\lambda,i}\to u_{\lambda}$ a.e. in $\Om$.
    In the case where $\Om$ is a bounded Lipschitz domain, we get these properties directly from
    Lemma \ref{lem:mollifying}, without the need for a cutoff function.
    Moreover, since $|u|\le M$, we can assume that also $|u_{\lambda,i}|\le M$
    (this follows from the constructions used in the proofs of
    Lemma \ref{lem:mollifying} and \ref{lem:appr with jump functions}).
    Notice that $u_{\lambda,i}=u_{\lambda}$ outside
    $B(0,R_{\lambda}+1)$,
    and so
    \begin{align*}
    &\left\{(x,y)\in \Om\times \Om\colon
	\frac{|u_{\lambda,i}(x)-u_{\lambda,i}(y)|}{|x-y|^{1+\gamma}}>\lambda+1\right\}
    \setminus \left\{(x,y)\in \Om\times \Om\colon
			\frac{|u_{\lambda}(x)-u_{\lambda}(y)|}{|x-y|^{1+\gamma}}>\lambda\right\}\\
            &\quad \subset B(0,R_{\lambda}+1)\times 
            B(0,(2M/(\lambda+1))^{1/(1+\gamma)}+R_{\lambda}+1)\\
             &\quad \quad\cup B(0,(2M/(\lambda+1))^{1/(1+\gamma)}+R_{\lambda}+1)\times B(0,R_{\lambda}+1).
    \end{align*}
    This justifies an application of Fatou's lemma to obtain % , for any $R>0$,
    \begin{align*}
    &\limsup_{i\to\infty} \nu_\gamma\left(\left\{(x,y)\in \Om\times \Om\colon
			\frac{|u_{\lambda,i}(x)-u_{\lambda,i}(y)|}{|x-y|^{1+\gamma}}>\lambda+1\right\}\right)\\
             &\quad \le  \nu_\gamma\left(\left\{(x,y)\in \Om\times \Om\colon
			\frac{|u_{\lambda}(x)-u_{\lambda}(y)|}{|x-y|^{1+\gamma}}>\lambda\right\}\right),
            %&\quad \le  \nu_\gamma\left(\left\{(x,y)\in \Om\times \Om\colon
			%\frac{|u_{\lambda}(x)-u_{\lambda}(y)|}{|x-y|^{1+\gamma}}>\lambda\right\}\right),
    \end{align*}
    %where $B_R \subset \R^n$ is the ball $B(0,R)$.
    For sufficiently large $i=i(\lambda)$, we have
    $\Vert u_{\lambda,i(\lambda)}-u_{\lambda}\Vert_{L^1(\Om)}<1/\lambda$,
    \[
    \left|\int_{\Om}E(Du_{\lambda,i(\lambda)})-\int_{\Om}E(Du_\lambda)\right|<1/\lambda,
    \]
    and
    \begin{equation}\label{eq:lambda plus one}
    \begin{split}
    & \nu_\gamma\left(\left\{(x,y)\in \Om\times \Om\colon
			\frac{|u_{\lambda,i(\lambda)}(x)-u_{\lambda,i(\lambda)}(y)|}{|x-y|^{1+\gamma}}>\lambda+1\right\}\right)\\
            &\quad \le \nu_\gamma\left(\left\{(x,y)\in \Om\times \Om\colon
			\frac{|u_{\lambda}(x)-u_{\lambda}(y)|}{|x-y|^{1+\gamma}}>\lambda\right\}\right)
            +1/\lambda^2.
    \end{split}
    \end{equation}
    For each $k\in \N$, let $v_{\lambda}:=u_{k,i(k)}$ for $\lambda\in [k+1,k+2)$.
    Then $v_{\lambda}\to u$ area-strictly, and
    for all $\lambda\in [k+1,k+2)$, we have
      \begin{align*}
    &\lambda \nu_\gamma\left(\left\{(x,y)\in \Om\times \Om\colon
			\frac{|v_{\lambda}(x)-v_{\lambda}(y)|}{|x-y|^{1+\gamma}}>\lambda\right\}\right)\\
                &\quad \le \lambda \nu_\gamma\left(\left\{(x,y)\in \Om\times \Om\colon
			\frac{|v_{\lambda}(x)-v_{\lambda}(y)|}{|x-y|^{1+\gamma}}>k+1\right\}\right)+1/\lambda\\
            &\quad \le \frac{\lambda}{k} k\nu_\gamma\left(\left\{(x,y)\in \Om\times \Om\colon
			\frac{|u_{k}(x)-u_{k}(y)|}{|x-y|^{1+\gamma}}
            >k\right\}\right)+1/\lambda
            \quad\textrm{by }\eqref{eq:lambda plus one}.
    \end{align*}
    As $\lambda\to \infty$, we have also $k\to\infty$ and $\lambda/k\to 1$.
    Thus $\{v_{\lambda}\}_{\lambda>0}$ is the required family.

    In the general case, we note that for the truncations we have
    $\min\{M,\max\{-M,u\}\}\to u$ in the area-strict sense in $\Om$ as $M\to\infty$ by the coarea formula. The chain rule for $\BV$ functions (Theorem 3.99 in \cite{AFP}) also implies
    \[ \abss{D^a u_M}(\Omega) \to \abss{D^a u}(\Omega), \quad \abss{D^s u_M}(\Omega)\to \abss{D^s u}(\Omega). \]
    Then it is straightforward to construct the required family.
    \end{proof}
    
		\begin{remark}\label{rmk:Gamma conv}
		Combining Lemma \ref{lem:appr with jump functions}
        with Theorem \ref{thm:SBV}, we can also show the following.
		Suppose $u\in \dot{\BV}(\Om)$. Then there exists a family
		$\{u_\lambda\}_{\lambda>0}$ with $u_\lambda\to u$ in $L^1_{\loc}(\Om)$ as $\lambda\to\infty$,
		and such that
		\begin{align*}
			\liminf_{\lambda\to \infty}F_{\gamma,\lambda}(u_\lambda,\Om)
			\le \limsup_{\lambda\to \infty}F_{\gamma,\lambda}(u_\lambda,\Om)
			\le \frac{C_n}{\gamma+1}|D u|(\Om).
		\end{align*}
	This shows that the ordinary $\Gamma$-limit of the functionals, if it exists, is at most
	\[
	\frac{C_n}{\gamma+1}|D u|(\Om),
	\]
	which in the case of $\SBV$ functions for which $D^a u\neq 0$ is strictly smaller than
	the ``pointwise'' limit given in \eqref{eq:SBV limit}.
	\end{remark}


\begin{thebibliography}{ACMM}

        \bibitem{AB}J.J. Alibert, G. Bouchitt\'{e}, 
        \textit{Non-uniform integrability and generalized Young measures},
        J. Convex Anal. 4 (1997), no. 1, 129-147.

        \bibitem{ADM}L. Ambrosio, G. Dal Maso, 
        \textit{On the relaxation in $BV(\Omega;\R^m)$ of quasi-convex integrals},
        J. Funct. Anal. 109 (1992), no. 1, 76-97.
		
		\bibitem{AFP}L. Ambrosio, N. Fusco, and D. Pallara,
		\textit{Functions of bounded variation and free discontinuity problems.}
		Oxford Mathematical Monographs. The Clarendon Press, Oxford University Press, New York, 2000.

		\bibitem{AGMP}C. Antonucci, M. Gobbino, M. Migliorini, and N. Picenni,
		\textit{Optimal constants for a nonlocal approximation of Sobolev norms and total variation},
		Anal. PDE 13 (2020), no. 2, 595--625.

        \bibitem{BB}A. Bj\"orn and J. Bj\"orn,
        \textit{Poincar\'e inequalities and Newtonian Sobolev functions on noncomplete metric spaces},
        J. Differential Equations 266 (2019), no. 1, 44--69.
        
		\bibitem{BSSY} H. Brezis, A. Seeger, J. Van Schaftingen, and P.-L. Yung,
		\textit{Families of functionals representing Sobolev norms},
		Anal. PDE 17 (2024), no. 3, 943--979.
		
		\bibitem{BSSY2} H. Brezis, A. Seeger, J. Van Schaftingen, and P.-L. Yung,
		\textit{Sobolev spaces revisited},
		Atti Accad. Naz. Lincei Rend. Lincei Mat. Appl. 33 (2022), no. 2, 413–437.
		
		\bibitem{BVSY} H. Brezis, J. Van Schaftingen, and P.-L. Yung,
		\textit{A surprising formula for Sobolev norms},
		Proc. Nat. Acad. Sci. U.S.A 118 (2021), no. 8, Paper No. e2025254118, 6 pp.

        \bibitem{EG} L. C. Evans and R. F. Gariepy, \textit{Measure theory and fine properties of functions.} Revised edition. Textbooks in Mathematics. CRC Press, Boca Raton, FL, 2015.

        \bibitem{FM}I. Fonseca, S. M\"{u}ller,
        \textit{Relaxation of quasiconvex functionals in $BV(\Omega,\R^p)$ for integrands $f(x,u,\nabla u)$},
        Arch. Rational Mech. Anal. 123 (1993), no. 1, 1-49.

        \bibitem{MaGo}M. Gobbino and N. Picenni,
        \textit{Gamma-liminf estimate for a class of non-local approximations of Sobolev and BV norms},
        J. Funct. Anal. 289 (2025), no. 9, Paper No. 111106, 25 pp.

	    \bibitem{KKSL13}J. Kinnunen, R. Korte, A. Lorent, and N. Shanmugalingam,
		\textit{Regularity of sets with quasiminimal boundary surfaces in metric spaces},
		J. Geom. Anal. 23 (2013), no. 4, 1607--1640.

        \bibitem{KR20} J. Kristensen, B. Rai\cb{t}\u{a}, \textit{An introduction to generalized Young measures.} Lecture note 45/2020. Max-Planck-Institut für Mathematik in den Naturwissenschaften Leipzig, 2020.
		
%		\bibitem{KrRind}J. Kristensen and F. Rindler,
%		\textit{Relaxation of signed integral functionals in BV},
%		Calc. Var. Partial Differential Equations 37 (2010), no. 1-2, 29--62.

		\bibitem{KrRind2}J. Kristensen and F. Rindler,
		\textit{Characterization of generalized gradient Young measures generated by sequences in
		$W^{1,1}$ and $\BV$},
		Arch. Ration. Mech. Anal. 197 (2010), no.2, 539--598.
		
		\bibitem{L-sharp}P. Lahti,
		\textit{A sharp lower bound for a class of non-local approximations of the total variation},
        Math. Ann. 392 (2025), no. 1, 469--486.
        
		\bibitem{Pic}N. Picenni,
		\textit{New estimates for a class of non-local approximations of the total variation},
		J. Funct. Anal. 287 (2024), no.1, Paper No. 110419, 21 pp.	
		
		\bibitem{Pol}A. Poliakovsky,
		\textit{Some remarks on a formula for Sobolev norms due to Brezis, Van Schaftingen and Yung},
		J. Funct. Anal. 282 (2022), no. 3, Paper No. 109312, 47 pp.

        \bibitem{Res}J.G. Reshetnyak,
        \textit{The weak convergence of completely additive vector-valued set functions},
        Sibirsk. Mat. Z. 9 (1968), 1386-1394.

        \bibitem{Rindler} F. Rindler,
        \textit{Calculus of variations.} Universitext. Springer International Publishing AG, Cham, 2018.
		
	\end{thebibliography}
\end{document}